\documentclass[12pt]{article}

\usepackage{graphicx}
\usepackage{amsmath,amsfonts,amssymb,amsthm,mathrsfs}
\usepackage{mathabx}
\usepackage[utf8]{inputenc}
\usepackage{bbm}
\usepackage{setspace}
\usepackage{bbm}
\usepackage{verbatim} 
\usepackage{xargs} 
\usepackage{authblk}
\usepackage[square,sort,comma,numbers]{natbib}
\usepackage{chngcntr}
\usepackage[font=small]{caption}
\usepackage
[draft=false,setpagesize=false,pdfstartview=FitH,
colorlinks=true,linkcolor=blue,citecolor=magenta,pagebackref=true,bookmarks=false]
{hyperref}

\usepackage{aliascnt}
\usepackage[capitalise,noabbrev,nameinlink]{cleveref}

\theoremstyle{theorem}
\newtheorem{theorem}{Theorem}[section]
\crefname{theorem}{Theorem}{Theorems}

\newaliascnt{lemma}{theorem}
\newtheorem{lemma}[lemma]{Lemma}
\aliascntresetthe{lemma}
\crefname{lemma}{Lemma}{Lemmas}

\newaliascnt{proposition}{theorem}
\newtheorem{proposition}[proposition]{Proposition}
\aliascntresetthe{proposition}
\crefname{proposition}{Proposition}{Propositions}

\newaliascnt{corollary}{theorem}

\aliascntresetthe{corollary}
\crefname{corollary}{Corollary}{Corollaries}

\theoremstyle{definition}
\theoremstyle{definition}
\newaliascnt{definition}{theorem}

\aliascntresetthe{definition}
\crefname{definition}{Definition}{Definitions}

\theoremstyle{remark}
\newtheorem{remark}{Remark}[]
\crefname{remark}{Remark}{Remarks}

\newaliascnt{example}{theorem}

\aliascntresetthe{example}
\crefname{example}{Example}{Examples}

\numberwithin{equation}{section}
\counterwithin{theorem}{section}
\counterwithin{figure}{section}

\crefname{theorem}{Theorem}{Theorems}
\crefname{lemma}{Lemma}{Lemmas}
\crefname{proposition}{Proposition}{Propositions}
\crefname{corollary}{Corollary}{Corollaries}
\crefname{definition}{Definition}{Definitions}
\crefname{remark}{Remark}{Remarks}
\crefname{subsection}{subsection}{subsections}

\usepackage{fullpage}

\newcommand{\R}{\mathbb{R}}
\newcommand{\N}{\mathbb{N}}
\newcommand{\e}{\varepsilon}
\newcommand{\cE}{\mathcal{E}}
\newcommand{\cK}{\mathcal{K}}
\newcommand{\cV}{\mathcal{V}}

\usepackage{braket,color}
\usepackage{physics2}
\usephysicsmodule{ab}
\usephysicsmodule{diagmat}
\usephysicsmodule{xmat}
\usepackage{derivative}
\usepackage[mathlines]{lineno} 
\usepackage{authblk}
\usepackage{enumitem}
\usepackage{braket}
\usepackage{mathtools}

\DeclareMathOperator{\supp}{supp}
\DeclareMathOperator{\dist}{dist}

\def\wt#1{\widetilde{#1}}

\crefname{pluralthoerems}{Theorems}{Theorems}

\newcommand{\affKEIO}{\emph{\small{ Keio University, Department of Mathematics, Faculty of Science and Technology,
Yagami Campus, 3-14-1 Hiyoshi, Kohoku-ku, Yokohama, Kanagawa 223-8522, Japan}}}
\newcommand{\affMIMUW}{\emph{\small{ University of Warsaw, Faculty of Mathematics, Informatics and Mechanics, 
ul. Banacha 2, 02-097 Warszawa, Poland}}}

\author[1]{Norihisa  Ikoma}
\author[2]{Krzysztof Myśliwy \footnote{Corresponding author, kmys@mimuw.edu.pl}}
\affil[1]{\affKEIO}
\affil[2]{\affMIMUW}

\title{Existence and order of the self--binding transition in non--local non--linear Schrödinger equations}

\begin{document}
\maketitle
\abstract{We consider a class of non--linear and non--local functionals giving rise to the Choquard equation with a suitably regular interaction potential, modelling, 
i.e., gases with impurities and axion stars. We study how existence of minimizers depends on the coupling constant, 
and find that there is a critical interaction strength needed for the minimizers to exist, both in dimensions two and three. 
In $d=3$, a minimizer exists also at the critical coupling but none do in $d=2$ under suitable assumptions on the potential. 
We also establish that in $d=3$ there exist other critical points beyond the global minimizer.}
\\

\\
\section{Introduction}

\subsection{Motivation}
\subsubsection{Bound states in classical and quantum dynamics}
Let $V:\mathbb{R}^d\rightarrow \mathbb{R}$ be an $L^{d/2}(\mathbb{R}^d)$ function with non--vanishing positive part $V^+:=\max \lbrace V,0\rbrace \not \equiv 0 $, 
and consider the basic problem of a point particle moving under the  influence of $(-V)$. 
Classically, if the initial position of the particle lies within the support of $V^+$ and its total energy $E$ is negative, 
then the trajectory of the particle is bounded in the support of $V^+$ 
and does not leave the set $\set{x | V(x)\geq -E}$: we say that the particle is \emph{trapped}. 

The quantum mechanical analogy of this trapping is the question of finding of what is called the \emph{bound states} of $-V$.  This corresponds to finding normalized $L^2$ minimizers of the functional 
\begin{equation}\label{scr-lin}
\mathcal{E}_s(u) \coloneq \int_{\mathbb{R}^d} \vab{\nabla u}^2 dx -\int_{\mathbb{R}^d} \vab{u}^2 V dx
\end{equation}
whose stationary points lead to the standard linear Schrödinger equation
\begin{equation*}\label{se}
-\Delta \phi-V \phi=e \phi
\end{equation*}
with $e:=\inf_{\Vab{u}_2=1}\mathcal{E}_s(u)$ being the \emph{ground state energy}.  A basic result is that if one knows a priori that $e<0$, 
then the minimizer necessarily exists (\cite[11.5 THEOREM]{LiebLoss} and \cite[Lemma 8.1.11]{Cazenave1}; 
thus, as in the classical case, a particle with negative total energy is trapped and cannot escape to infinity, since its stationary state is a genuine $L^2$ function, the minimizer $u$. 

In contrast to the classical dynamics in which one can always ensure the trapping for an arbitrary $V$ of the type given simply by setting the initial position and velocity accordingly, it is not always possible in the quantum case; for example in $d=3$, given $V$ sufficiently small in $L^{3/2}(\mathbb{R}^3)$ norm, one has $\mathcal{E}_s(u)>0$ for all $u$ with $\Vab{u}_2=1$ (e.g., by the Sobolev inequality) while there is a sequence of $L^2$ normalized functions for which $\mathcal{E}_s(u_j)\rightarrow 0$; consequently, a minimizer does not exist, and no binding takes place. For instance, in the case of two helium atoms, their corresponding interaction potential $V$ comes in with an $L^{3/2}$ norm that is too small to support a bound state between the two, which explains why helium does not crystallize even at extremely low temperatures \cite{Huang}.



Interestingly, the situation suggested by classical intuition is restored if the motion of the particle is constrained to lower dimensions; in $d=1$ and $d=2$, the functional \eqref{scr-lin} \emph{always} has a minimizer with $e<0$, provided that $\int V > 0$, in particular, $V^+\neq 0$, no matter how small $V$ is \cite{Simon}. 

For a succinct summary, let us fix $V: \mathbb{R}^d\to \mathbb{R}$ with $\int_{\mathbb{R}^d} V dx >0$ and a number $g>0$, and consider
\begin{equation}\label{scr}
	\begin{aligned}
		&\mathcal{E}_s(u;g) \coloneq \int_{\R^d} \vab{\nabla u}^2 dx -g\int_{\R^d} \vab{u}^2 V dx \quad 
		\text{for $u \in H^1(\R^d)$},
		\\
		&e_g \coloneq \inf \Set{ \mathcal{E}_s (u;g) | u \in H^1(\R^d) , \ \Vab{u}_2 = 1 }.
	\end{aligned}
\end{equation} 
Then, under suitable summability conditions on $V$, we have
\begin{enumerate}
\item 
when $d=1$ and $d=2$, for any $g > 0$ there exists $\varphi_g \in H^1(\R^d)$ such that 
\begin{equation*}
\Vab{\varphi_g}_2=1,\quad  \mathcal{E}_s(\varphi_g;g)=e_g <0;
\end{equation*}

\item 
when $d\geq 3$, there exists $g^*$ such that $e_g < 0$ holds if and only if $g>g^*$ and 
$e_g$ is never attained for all $g < g^*$. 
We say that a \emph{binding transition} takes place in this case. In the language of the theory of phase transitions, one can say that $d=3$ is the \emph{lower critical dimension} for the binding transition to occur.
\end{enumerate}

\subsubsection{Nonlinear binding: Brief description of main results}
The purpose of the present paper is to investigate how the above situation changes in case where the \emph{linear, local} model given by \eqref{scr} is replaced by a \emph{non--linear, non--local} functional known as the Pekar (or Hartree) functional: 
\begin{equation}\label{epek}
\mathcal{E}_g(u) \coloneq \frac{1}{2}\int_{\mathbb{R}^d} \vab{\nabla u(x)}^2 d x
-\frac{g}{4} \iint_{\mathbb{R}^d\times \mathbb{R}^d} \vab{u(x)}^2V(x-y)\vab{u(y)}^2 dx dy. 
\end{equation}
The critical points of  $\mathcal{E}_g$, subject to the condition $\Vab{u}_2=1$, lead to the \emph{inhomogeneous Choquard's equation} 
\begin{equation}\label{choq}
-\Delta u - \frac{g}{2} \ab(|u|^2\ast V + |u|^2 \ast \check{V} )u = \lambda u \quad \text{where} \ \check{V} (x) \coloneq V(-x) 
\end{equation}
for $\lambda \in \mathbb{R}$ and the $\ast$ denotes convolution. 
Physically, one should think of this model as a description of situations where the surroundings of the particle act back on it in the form of a potential whose exact form depends on the state of the particle, $u$. Such scenarios are quite common in nature, and we shall give a few examples below. 
Notice that \eqref{choq} admits the translation invariance, not present in the Schrödinger equation with potential, 
and this reflects the fact that the underlying system (particle+surroundings) is translationally invariant. 

Now, under suitable assumptions on $V$, we shall prove the following (see \cref{thm,thm:gstar}):
\begin{enumerate}
\item For \emph{both} $d=2$ and $d=3$, there exists a \emph{critical} $g^* > 0$ such that $\mathcal{E}_g$ has no normalized $L^2$ minimizer for $g<g^*$ and does have at least one for $g>g^*$, while a minimizer exists when $d=1$ for all $g>0$; thus, the introduction of non--linearity and non--locality into the model \emph{shifts the lower critical dimension for the binding transition to occur from $d=3$ to $d=2$} as compared to the Schrödinger case \eqref{scr};
\item in $d=3$, a minimizer also exists for the corresponding critical $g^*$ while it does \emph{not} for $d=2$; in analogy to usual phase transitions, we thus say that the binding transition is of \emph{first order} in $d=3$ and of \emph{second order} in $d=2$.
\end{enumerate}
Below, we shall explain in more detail why we chose to refer to the above phenomenon as a first or second order transition.

It is enough to think of the first order transition in the sense given above as  simply to the existence of a minimizer on which the functional evaluates to zero, as it does for the zero function. The intuitive picture of a continuous \emph{function}  $f: [0,\infty) \to \mathbb{R}_{+}$ with $f(0)=0$ having the same property suggests that another critical point, beyond the minimizer, with a \emph{strictly positive} value of the functional should also exist. 
This is what we prove in \cref{thm:2ndsol-nonex}, in fact we prove that such \emph{metastable} solutions exist for all $g$ large enough, 
even for those smaller than $g_3^*$, in accordance with the intuition from thermal physics. This result not only strengthens the picture of first order transition but also has possible physical implications in the context of axion stars, as we briefly explain below.


Before stating the main result in a more precise manner, let us provide additional motivation for the study of this particular model, and why the information on how the trapping transition is realized is of importance. 
\subsubsection{Background}

The functional \eqref{epek} appears in many physical contexts. Two mathematical extremes are the case of the delta interaction, $V(x)=\delta(x)$ which leads to the \emph{Gross--Pitaevskii} functional known from the theory of Bose--Einstein condensation for interacting Bose gases (here, one usually has $g<0$), and the case of the Coulomb interaction, $V(x)=\frac{1}{|x|}$ with $g>0$.  A mixture of these two, the \emph{Gross-Pitaevskii-Poisson functional} has been applied to model \emph{gravitationally bound axion stars}, which are among the candidates for dark matter. Note that existence/non--existence of a minimizer translates to the stability of such objects, which, in turn, is believed to have observable consequences. 
We refer to \cite{Cazenave1,CazenaveLions,Li76,Li80,Lions1,MorozVanShaftingen,Moll23,Stuart,Ye16} 
for a recent mathematical study of a related functional in this context.



In our study, we are interested in a $V$ that lies beyond the two extremes mentioned, in particular, we impose certain summability conditions on $V$. Two important applications of such models are the following. First, they appear in the study of self--bound axion star dark matter where the binding of these structures can ensue with the intrinsic long--range attraction of the axions, with gravity effects negligible \cite{braatenzhang}. 
Secondly, the functional \eqref{epek} has recently resurfaced in the study of ionic atoms immersed in Fermi gases \cite{MJ24}. In this context, the exact form of $V$ depends on the details of how the atom--ion interactions are modelled. One important example is $V_b(x)=\frac{1}{(|x|^2+b^2)^2}$ which is a regularized potential having the $|x|^{-4}$ decay typical to ion--induced dipole interactions. In principle, the parameter $b$ can be imagined to be subject to experimental control using Feshbach resonances, and as such translates to out $g$. 
For such a class of $V$, Lions \cite{Li80,Lions1} treated the existence of minimizer. 
More precise results will be given in \cref{Rem1}. 


\subsection{Main result}

Throughout this paper, we always assume $d=1,2,3$. 
For our results, we list assumptions on $V$: 

\begin{enumerate}[label=(A\arabic*)]
	\item \label{A1}
	When $d=1$, $V \in L^1(\R)$; when $d=2,3$, $V \in L^{d/2} (\R^d)$ and 
	for each $\e \in (0,1)$ there exist $q_\e \in (d/2,\infty]$, $V_{\e,1}, V_{\e,2}$ such that 
	$V= V_{\e,1} + V_{\e,2}$, $V_{\e,1} \in L^{q_\e} (\R^d)$ and $\Vab{V_{\e,2}}_{d/2} < \e$;

	\item \label{A2}
	there exists $r_0>0$ such that $\displaystyle \inf_{B(0,r_0)} V \eqcolon V_0 > 0$; 
\end{enumerate}

\begin{enumerate}[label=(B\arabic*)]
	\item \label{B1}
	$V \in C^1(\R^d)$;
	\item \label{B2}
	the function $V$ is radial and non-increasing in $[0,\infty)$; 
	\item \label{B3}
	$V \in L^1(\R^d)$ and 
	$W(x) \coloneq x \cdot \nabla V (x) \in L^1(\R^d)$; 
	\item \label{B4}
	there exists an $r^*>0$ such that 
	$V(x) + W(x)/2 = 0$ is equivalent to $\vab{x} = r^*$. 
\end{enumerate}
For $g>0$ and $m>0$, we set 
\begin{equation}\label{minpro}
e(g,m) \coloneq \inf \Set{ \mathcal{E}_g(u) | u \in S_m }, \quad 
S_m \coloneq \Set{ u \in H^1(\R^d) | \Vab{u}_2^2 = m }.
\end{equation}

Our first result is to prove the existence of minimizers corresponding to $e(g,m)$ when $e(g,m) < 0$: 
\begin{theorem}\label{thm}
Let $d\in \set{ 1, 2, 3}$, $m>0$, {\rm \ref{A1}} and {\rm \ref{A2}} hold. 
Suppose $V \geq 0$ when $d=1$. 
Then there exists $g^*_d (m) \geq 0$ such that $g_d^*(m) = 0$ $(d=1)$, 
$g_d^*(m) > 0$ $(d=2,3)$, $e(g,m) < 0$ if and only if $g>g_d^*(m)$, 
and the minimizing problem \eqref{minpro} admits a minimizer for every $g>g^*_d(m)$. 
Furthermore, if $0<g<g^*_d(m)$, then \eqref{minpro} has no minimizer. 
\end{theorem}

\begin{remark}\label{Rem1}
	\begin{enumerate}
		\item 
		When $d=1$, if the condition $V \geq 0$ is dropped, then 
		there exists $V$ with \ref{A1} and \ref{A2} such that $g_1^*(m) > 0$ holds for all $m > 0$ (see \cref{Sec7}). 
		\item 
		In \cite{Li80,Lions1}, the existence of minimizers was shown under the assumption $e(g,m) < 0$ and $V \in L^p_w(\R^d) + L^q_w(\R^d)$ 
		with $p,q > d/2$ if $d \geq 3$ and $p,q > 1$ if $d=1,2$. 
	\end{enumerate}
	
\end{remark}

We next investigate \eqref{minpro} in the case $g=g^*_d(m)$:

\begin{theorem}\label{thm:gstar}
	Let $m>0$ and suppose {\rm \ref{A1}} and {\rm \ref{A2}}. 
	\begin{enumerate}
		\item 
		When $d=3$ and $g=g^*_3(m)$, the minimizing problem \eqref{minpro} has a minimizer. 
		\item 
		When $d=2$ and $g=g^*_2(m)$, assume also {\rm \ref{B1}--\ref{B4}}. 
		Then \eqref{minpro} admits no minimizers. 
	\end{enumerate}
\end{theorem}

\begin{remark}\label{rem:exam}
	\begin{enumerate}
		\item 
		Conditions \ref{A1}, \ref{A2} and \ref{B1}--\ref{B4} when $d=2$ are verified by 
		the regularized ion--atom potential $V(x) = 1/ (\vab{x}^2 + b^2)^2$. 
		In fact, direct computations imply 
		\[
		W(x) = - \frac{4\vab{x}^2}{(\vab{x}^2 + b^2)^3}, \quad 
		V(x) + \frac{W(x)}{2} = \frac{b^2-\vab{x}^2}{(\vab{x}^2+b^2)^3}.
		\]
		Thus, for any $b>0$, \ref{A1}, \ref{A2} and \ref{B1}--\ref{B4} hold. 
		\item 
		When $d=1$ and $g^*_d(m) > 0$, the existence of minimizers is not straightforward and remains open. 
		To our knowledge, the problem of characterizing the critical $g$ necessary for a bound state to arise if $\int V<0$ in $d=1$ and $d=2$ appears not to have been fully resolved \emph{even in the linear case. }
		
		\item 
		For a local nonlinearity, a similar existence result at the threshold mass was obtained in \cite{Sh14}. 
	\end{enumerate}
	
\end{remark}

Our final result concerns the existence and non--existence of \emph{metastable states} in $d=3$. This physical notion corresponds to local minimizers and other critical points of the functional. As discussed below, their existence for some $0<g<g^{*}_3(m)$ is to be expected based on the classification of the energy landscape in terms of a first order transition. Similarly, for small enough couplings, one expects that no critical points exist. In this context, we are able to show that for sufficiently small $g$, there exist no critical points which are either radial or strictly positive. 
It is worth mentioning that the existence of local minimizer for a local nonlinearity was proved in \cite{JeLu22a}. 

\begin{theorem}\label{thm:2ndsol-nonex}
	Let $d=3$, $m>0$ and assume {\rm \ref{A1}}, {\rm \ref{A2}}.
	\begin{enumerate}
		\item 
		Suppose $W (x) = x \cdot \nabla V(x) = W_1 + W_\infty \in L^1(\R^3) + L^\infty(\R^3)$ with $W_\infty(x) \to 0$ as $\vab{x} \to \infty$.  Assume that $V$ is radial with $V \geq 0$ in $\R^3$. 
		Then there exists $g_1 > 0$ such that 
		for all $g \in ( g_1 , \infty )$ 
		the functional $\cE_g|_{S_m}$ has a critical point $u \in S_m$ with $\cE_g(u) > 0$. 
		Furthermore, if $V$ is non-increasing, then the constant $g_1$ can be chosen so that $g_1< g^*_3(m)$. 
		
		\item 
		There exists $g_3 \in (0,g^*_d(m))$ and $\rho_0 > 0$ such that for all $g \in (g_3, g^*_d(m) )$,	the problem 
	\begin{equation}\label{loc-minpro}
		\wt{e} (g,m) \coloneq \inf \Set{ \cE_g (u) | u \in S_m, \ \Vab{\nabla u}_2^2 > \frac{\rho_0^2}{4}  }
	\end{equation}
	admits a minimizer. Note $\wt{e} (g,m) > 0$ holds since $g<g^*_d(m)$ and there is no minimizer for $e(g,m) = 0$. Thus, the functional $\cE_g|_{S_m}$ admits a local minimizer, corresponding physically to a \emph{metastable state} of the problem.

		\item 
		Finally, suppose {\rm \ref{B1}} with  $W (x) \coloneq x \cdot \nabla V(x) \in L^{3/2} (\R^3)$. 
		Then there exists $g_2 \in (0,\infty)$ such that 
		for any $g \in (0,g_2)$ the functional $\cE_g|_{S_m}$ does not admit any critical point $u$ with $u \geq 0$ (resp. $u \leq 0$)
		or $u(x) = u(\vab{x})$. 
	\end{enumerate}
\end{theorem}

\begin{remark}
	The critical point $u$ in \cref{thm:2ndsol-nonex} 1 is obtained via the mountain pass theorem for $\cE_g|_{S_m}$. 
	Therefore, $u$ may be different from the local minimizer $v$ of \eqref{loc-minpro}. 
	In fact, if $g<g^*_3(m)$ is close to $g^*_3(m)$, then $\cE_g(v) < \cE_g(u)$ holds. 
\end{remark}

\begin{remark}
By \cref{thm:gstar,thm:2ndsol-nonex}, 
under suitable assumptions on $V$, when $d=3$ (resp. $d=2$), a minimizer exists if and only if $g \geq g_3^*(m)$ (resp. $g>g_2^*(m))$. 
The reason behind referring to this distinction as a ``first order" or ``second order" transition, respectively, lies in the analogy to the usual phase transitions in statistical physics. At a first order transition, distinct stable phases change discontinuously one into the other: think of the usual boiling of water, where the density abruptly becomes very small as one reaches the boiling point. If the transition is of second order, the phases change continuously: an example is provided by the process of demagnetizing a piece of iron upon heating.  In our case, the transition is of first order in $d=3$ as we have two \emph{coexisting} classes of minimizing sequences at the critical value $g_3^*(m)$, those converging to a genuine minimizer, corresponding to a bound state, and those converging weakly to zero, corresponding to a free particle with vanishing kinetic energy. In $d=2$, under the additional assumptions, one always has either of the two scenarios happening, with their coexistence at any $g$ excluded. 

\end{remark}

\begin{remark}
The first order character of the transition in $d=3$ is further strengthened by \cref{thm:2ndsol-nonex}. 
In statistical physics, first order transitions are accompanied by the existence of \emph{metastable} states, which are local, rather than global equilibrium states of the system (a very good qualitative discussion can be found in \cite[Chapter 9]{Callen}).
These correspond to stationary states of the Gibbs functional other than its global minimum
and start to appear away from the transition point: one common example is \emph{supercooled water}. In our case, these metastable states are the local minimizers whose existence is established in \cref{thm:2ndsol-nonex}. 
 This picture is in line with the existence of positive energy critical points in \cref{thm:2ndsol-nonex}.

\end{remark}

\begin{remark}
In the context of axion stars, the spontaneous decay of metastable bound structures corresponding to the positive energy solutions of \cref{thm:2ndsol-nonex}  into unbound structures should bear observable consequences as the positive binding energy is released. We are unable to provide bounds on this energy, given by a suitable mountain pass critical point value, and leave it as an open problem. 
\end{remark}

\begin{remark}
	According to \cref{thm}, when $d=1$ and $V \in L^1(\R)$ satisfies $V \geq 0$, 
	there is always a minimizer. 
	Thus, $d=2$ is the \emph{lower critical dimension} for the transition to occur in the case of \eqref{epek} within the class of purely attractive potentials 
	(meaning $V\geq 0 $ and $V\neq 0$), in contrast to \eqref{scr}, where the lower critical dimension for this class is $d=3$.
\end{remark}

Finally, we comment on ideas and difficulties of the proofs of \cref{thm,thm:gstar,thm:2ndsol-nonex}. 
The proof of \cref{thm} is standard and 
we employ the concentration compactness principle due to \cite{Lions1,Lions2}. 
Here a slightly modified argument in \cite{Cazenave1} is used and 
a similar argument can be found in \cite{JeLu20,JeLu22a}. 
The inequality $e(g,m) < 0$ is useful to find a minmizer and to prove it, \ref{A2} and $g > g_d^*(m)$ play a role. 
The positivity of $g_d^*(m)$ is a consequence of Young's and Sobolev's inequality ($d=3$) and Young's and 
the Gagliardo--Nirenberg inequality ($d=2$).

For the proof of \cref{thm:gstar} when $d=3$, 
we cannot argue as in the proof of \cref{thm} since $e( g^*_3(m) , m ) = 0$ holds. 
Here we choose a sequence $\{g_j\}_j$ so that $g_j \searrow g^*_3(m)$ and 
by \cref{thm}, let $u_j \in S_m$ be a minimizer for $e(g_j,m) < 0$. 
Then $\{u_j\}_j$ is a minimizing sequence for $e(g^*_3(m),m) = 0$. 
To avoid the vanishing of $\{u_j\}_j$, we shall apply the Lieb-Thirring (the Cwikel-Lieb-Rozenblum) inequality 
for operators $-\Delta - g(V * u_j^2)/2  $, which admits at least one negative eigenvalue. 
We believe this approach nicely combines operator theory and nonlinear analysis. 
It is worth to noting that \cite{MeSz} uses the same inequality in a different fashion. 
Once the vanishing is ruled out, an argument similar to that of \cref{thm} is used to find a minimizer.  
On the other hand, in the case $d=2$, the above approach breaks down and 
we will prove the nonexistence of minimizers. 
This task is delicate and to achieve this goal, 
we first establish a suitable non--local Pokhozaev--type identity (\cref{p:Poho}) for any 
critical point of $\cE_{g}|_{S_m}$ with nonpositive energy. 
Here we closely follow the arguments in \cite[section 3]{MorozVanShaftingen13}. 
Then we argue by contradiction and 
combine the identity with \ref{B1}--\ref{B4} and the monotonicity of the minimizers to get a contradiction. 
As pointed outed in \cref{rem:exam}, our conditions include some important potentials, however,  
we do not know to which extent conditions on $V$ are necessary, and leave the general case as an open problem.

The proof of the solvability of the local minimizing problem in \cref{thm:2ndsol-nonex} is based on ideas from \cite{JeLu22a}. 
The nonexistence result in \cref{thm:2ndsol-nonex} 
is proved by combining the non-local Pokhozaev-type identity with Young's and Sobolev's inequality. 
In this case, the non-local Pokhozaev-type identity for critical point $u$ of $\cE_g|_{S_m}$ with $\pm u \geq 0$ or $u(x) = u(\vab{x})$ is necessary. 
For this purpose, the decay estimate of critical points becomes important and this will be done in \cref{l:decay}. 
On the other hand, to find a critical point of $\cE_g|_{S_m}$ with positive energy, 
the mountain pass theorem for $\cE_g|_{S_m}$ is employed. 
For our problem, eliminating the vanishing of Palais-Smale sequence is crucial. 
To this end, we will use an augmented functional of $\cE_{g}$ based on scaling and 
this functional is utilized in other papers, for instance see \cite{Jeanjean,BMRV21,ikomatanaka,MoRiVe22}. 
Recall that the augmented functional usually plays a role to find a bounded Palais-Smale sequence, 
while our usage is different since this is exploited to exclude the vanishing of Palais-Smale sequence.

\section{Preliminaries}\label{sec2}

In this section, we prove several auxiliary properties of $g \mapsto e(g,m)$ and $m \mapsto e(g,m)$. 
For brevity, we denote 
\begin{equation*}
\mathcal{K}(u) \coloneq \frac{1}{2} \Vab{\nabla u}_2^2, 
\quad \mathcal{V}(u) \coloneq \frac{1}{4}\iint_{\mathbb{R}^d\times \mathbb{R}^d} \vab{u(x)}^2V(x-y)\vab{u(y)}^2 dx dy
\end{equation*}
throughout, with the short--hand for the functional:
\begin{equation*}
	\mathcal{E}_g(u)=\mathcal{K}(u)-g\mathcal{V}(u).
\end{equation*}
Notice that $V \in L^1(\R)$ ($d=1$) and $V \in L^{d/2} (\R^d)$ ($d=2,3$) by \ref{A1}. 
Hence Young's inequality leads to 
\begin{equation}\label{basineqcV}
	4\vab{\cV(u)} \leq 
	\begin{dcases}
		\Vab{u^2}_2^2 \Vab{V}_1 \Vab{u^2}_2^2 = \Vab{u}_4^4 \Vab{V}_1 & \text{if $d=1$},\\
		\Vab{u^2}_{d/(d-1)} \Vab{V}_{d/2} \Vab{u^2}_{d/(d-1)} 
		= \Vab{V}_{d/2} \Vab{u}_{2d/(d-1)}^4 & \text{if $d =2,3$}.
	\end{dcases}
\end{equation}
In particular, $\cV \in C^1(H^1(\R^d), \R)$ and 
\[
\begin{aligned}
	\cV'(u) \varphi &= \frac{1}{2} \int_{\R^d \times \R^d} u(x) \varphi(x) V(x-y) u(y)^2 + u^2(x) V(x-y) u(y) \varphi(y) dxdy
	\\
	&= \int_{\R^d \times \R^d} u (x) V(x-y) u^2 (y) \varphi(x) dx dy \quad \text{if $V(x) = V(-x)$}.
\end{aligned}
\]
Moreover, if $\{u_j\}_j$ is bounded in $H^1(\R^d)$ and $\Vab{u_j}_q \to 0$ for some $q \in (2,2^*)$ 
where $2^* \coloneq \infty$ ($d=1,2$) and $2^* \coloneq 2d/(d-2)$ ($d=3$), then since $\Vab{u_j}_r \to 0$ for any $r \in (2,2^*)$ 
by the interpolation, \eqref{basineqcV} gives 
\begin{equation}\label{vanicV}
	\cV(u_j) \to 0 \quad \text{as $j \to \infty$}. 
\end{equation}

\begin{proposition}\label{prop1}
	Under {\rm \ref{A1}}, $e(g,m)>-\infty$ holds for all $g\geq 0$ and $m>0$. 
	Moreover, for every $\{u_j\}_j \subset S_m$, if $\{\cE_g(u_j) \}_j$ is bounded, then $\{u_j\}_j$ is bounded in $H^1(\mathbb{R}^d)$. 
	As a consequence, all minimizing sequences for $e(g,m)$ are bounded in $H^1(\R^d)$. 
\end{proposition}
\begin{proof}
In what follows, let $u \in S_m$. 
Recall the Gagliardo--Nirenberg inequality: 
\begin{equation}\label{GN-ineq}
	\Vab{v}_p \leq C \Vab{\nabla v}_2^\theta \Vab{v}_2^{1-\theta} \quad 
	\text{for all $v \in H^1(\R^d)$ where $\frac{1}{p} = \theta \ab( \frac{1}{2} - \frac{1}{d} ) + \frac{1-\theta}{2}$. }
\end{equation}
When $d=1$, it follows from \eqref{basineqcV} and \eqref{GN-ineq} that 
\begin{equation*}
	\vab{\cV(u)} \leq C \Vab{V}_1 \Vab{u}_4^4 \leq C \Vab{V}_1 \Vab{\nabla u}_2 \Vab{u}_2^{3} = C m^{3/2} \Vab{V}_1 \Vab{\nabla u}_2.
\end{equation*}
Hence, for each $u \in S_m$, 
\begin{equation}\label{LBEg-1d}
	\cE_g(u) \geq \frac{1}{2} \Vab{\nabla u}_2^2 - C g m^{3/2} \Vab{V}_1 \Vab{\nabla u}_2. 
\end{equation}
Thus, $e(g,m) > -\infty$ holds and the other assertions also hold.

When $d=2,3$, for $\e \in (0,1)$ chosen later, by \ref{A1}, there exist $q_{\e} \in (d/2,\infty]$, $V_{\e,1} \in L^{q_\e} (\R^d)$ and $V_{\e,2} \in L^{d/2} (\R^d)$ such that 
\[
V = V_{\e,1} + V_{\e,2}, \quad \Vab{V_{\e,2}}_{d/2} < \e. 
\]
Decompose $\cV$ as 
\[
\cV(u) = \cV_{\e,1} (u) + \cV_{\e,2}(u), \quad 
\cV_{\e,i}(u) \coloneq \frac{1}{4} \iint_{\R^d \times \R^d} u^2(x) V_{\e,i} (x-y) u^2(y) dx dy.
\]
For $p=2d/(d-1)$, $\theta$ in \eqref{GN-ineq} fulfills $4\theta = 2$ and consequently \eqref{basineqcV} yields 
\begin{equation}\label{bdd-2ndterm}
	\vab{\cV_{\e,2} (u)} \leq \frac{1}{4} \Vab{V_{\e,2}}_{d/2} \Vab{u}_{2d/(d-1)}^4 \leq C \Vab{V_{\e,2}}_{d/2} \Vab{\nabla u}_2^2 \Vab{u}_2^{ 2 } 
	< C m \e \Vab{\nabla u}_2^2. 
\end{equation}
Fix $\e \in (0,1)$ so that $gC m \e < 1/4$. By $V_{1,\e} \in L^{q_\e} (\R^d)$ with $q_\e > d/2$, Young's inequality implies 
\[
4\vab{\cV_{\e,1}(u)} \leq \Vab{V_{\e,1}}_{q_\e} \Vab{u^2}^2_{ 2q_\e/(2q_\e-1) }
= \Vab{V_{\e,1}}_{q_\e} \Vab{u}_{4q_\e/(q_\e-1)}^4. 
\]
In \eqref{GN-ineq} with $p=4q_\e/(2 q_\e-1)$, we see that $4\theta = d/q_\e < 2$, and hence 
\begin{equation}\label{bdd-1stterm}
	\vab{\cV_{\e,1}(u)} \leq C \Vab{V_{\e,1}}_{q_\e} \Vab{\nabla u}_2^{ d/q_\e } \Vab{u}_2^{ 4 - d/q_\e } 
	= C \Vab{V_{\e,1}}_{q_\e} m^{ 2 - d/(2q_\e) } \Vab{\nabla u}_2^{d/q_\e}.
\end{equation}
By \eqref{bdd-2ndterm} and \eqref{bdd-1stterm}, the choice of $\e$ gives 
\begin{equation*}
	\vab{\cV(u)} \leq C \Vab{V_{\e,1}}_{q_\e} m^{ 2 - d/(2q_\e) } \Vab{\nabla u}_2^{d/q_\e}  + \frac{1}{4} \Vab{\nabla u}_2^2
	\quad \text{for all $u \in S_m$}. 
\end{equation*}
Since $d/q_\e < 2$, we conclude 
\begin{equation}\label{LBEg}
\cE_g(u) \geq \frac{1}{4} \Vab{\nabla u}_2^2 - C \Vab{V_{\e,1}}_{q_\e} m^{ 2 - d/(2q_\e) } \Vab{\nabla u}_2^{d/q_\e} \quad 
\text{for any $u \in S_m$}
\end{equation}
and the desired results hold. 
\end{proof}

\begin{lemma}\label{cont_mon}
	Let {\rm \ref{A1}} hold. Then for all $m>0$, $e(g,m) \leq 0$ and the function $ g\mapsto e(g,m)$ is non-increasing and continuous.
\end{lemma}
\begin{proof}
We first show that for all $g\geq 0$ and $m>0$, it holds $e(g,m)\leq 0$. 
Let $\varphi \in C^\infty_c(\R^d)$ with $\Vab{\varphi}_2^2 = m$ and for $t>0$, consider $\varphi_t \coloneq t^{d/2} \varphi(t\cdot)$. 
Notice that $\Vab{\varphi_t}_2^2 = \Vab{\varphi}_2^2 = m$ and write $V_t \coloneq V(\cdot/ t)$. 
Then
\begin{equation}\label{est-egm1}
	\begin{aligned}
		e(g,m) \leq \mathcal{E}_{g}(\varphi_t) 
		&= t^2 \cK(\varphi)-\frac{g}{4} \iint_{\R^d \times \R^d} \vab{\varphi(x)}^2 V_t(x-y) \vab{\varphi(y)}^2 dx dy
	\end{aligned}
\end{equation}
and \eqref{basineqcV} yields 
\[
\iint_{\R^d \times \R^d} \vab{\varphi (x)}^2 \vab{ V_{t} (x-y) } \vab{\varphi(y)}^2 dx dy 
\leq 
\begin{dcases}
	\Vab{V_t}_1 \Vab{u}_4^4 & \text{if $d=1$}, \\
	\Vab{V_{t}}_{d/2} \Vab{\varphi}_{2d/(d-1)}^4 & \text{if $d=2,3$}.\\
\end{dcases}
\]
By $\Vab{V_t}_1 \to 0$ ($d=1$) and $\Vab{V_t}_{d/2} \to 0$ ($d=2,3$) as $t \to 0^+$, 
it follows from \eqref{est-egm1} that 
$e(g,m) \leq \lim_{t \to 0^+} \cE_g(\varphi_t) = 0$.

Now let $m>0$ and $g_2\geq g_1$. Let $\lbrace u^{(g_1)}_i\rbrace_i \subset S_m$ be a minimizing sequence for $\mathcal{E}_{g_1}$. 
Note that since $e(g_1,m)\leq 0$, we have $\liminf_{i\rightarrow \infty} \mathcal{V}(u_i^{(g_1)})\geq 0$ and thus
\begin{equation*}
e(g_2,m)\leq \limsup_{i\rightarrow \infty} \mathcal{E}_{g_2}(u_i^{(g_1)}) 
=\limsup_{i\rightarrow \infty}\left(\mathcal{E}_{g_1}(u_i^{(g_1)})-(g_2-g_1)\mathcal{V}(u_i^{g_1})\right)\leq e(g_1,m),
\end{equation*}
so that $g\mapsto e(g,m)$ is non-increasing.

Next, we show the left continuity of $g\mapsto e(g,m)$. Let $g_0> 0$ be given.  
For any $\varepsilon>0$ we can find a function $\varphi_{\varepsilon} \in S_m$ such that $\mathcal{E}_{g_0}(\varphi_{\varepsilon})\leq e(g_0,m)+\varepsilon$.  
Therefore, $\lim_{g\rightarrow g_0^-} \mathcal{E}_g(\varphi_{\varepsilon})= \cE_{g_0} (\varphi_\e) \leq e(g_0,m)+\varepsilon$. 
The monotonicity of $e(g,m)$ in $g$ implies  
\begin{equation*}
e(g_0,m) \leq \lim_{g\rightarrow g_0^-} e(g,m) \leq \lim_{g \to g^-_0} \cE_g(\varphi_\e) \leq e(g_0,m)+\varepsilon.
\end{equation*}
Since $\e>0$ is arbitrary, $\lim_{g\rightarrow g_0^-} e(g,m)= e(g_0,m)$ holds.

To show the right continuity, consider any sequence $\lbrace g_n \rbrace_{n=1}^{\infty}$ with $g_n\geq g_0$  and $\lim_{n\rightarrow\infty}g_n =g_0$. For any $n$, we can choose $u_n\in S_m$ such that 
\begin{equation}\label{Egnun}
e(g_n,m)\leq \cE_{g_n} (u_n) \leq e(g_n,m)+\frac{1}{n} \leq \frac{1}{n}. 
\end{equation}
Using the bounds \eqref{LBEg-1d} and \eqref{LBEg}, we see that $\lbrace u_n \rbrace_{n=1}^\infty$ is bounded in $H^1(\R^d)$. 
By \eqref{basineqcV}, $\lbrace \mathcal{V}(u_n)\rbrace_{n=1}^{\infty}$ is also bounded, and $\lim_{n\rightarrow\infty}(g_n-g_0)\mathcal{V}(u_n)=0$. 
Hence, the monotonicity of $e(g,m)$ and \eqref{Egnun} give 
\begin{equation*}
\begin{aligned}
\lim_{n\rightarrow \infty} e(g_n,m) &\leq  e(g_0,m)
\leq \liminf_{n\rightarrow \infty} \mathcal{E}_{g_0}(u_n)= \liminf_{n\rightarrow \infty} \ab(\mathcal{E}_{g_n}(u_n)
+(g_n-g_0) \mathcal{V}(u_n) )\\ 
&\leq \liminf_{n\rightarrow \infty} \ab(e(g_n,m)+\frac{1}{n}+(g_n-g_0) \mathcal{V}(u_n) )=\lim_{n\rightarrow \infty} e(g_n,m), 
\end{aligned}
\end{equation*}
which demonstrates the right continuity and concludes the proof of the Lemma.
\end{proof}

\begin{lemma}\label{l:gstar}
	Suppose {\rm \ref{A1}} and {\rm \ref{A2}}. 
	Then for every $m>0$, there exists $g^*_d(m) \geq 0$ such that 
	$e(g,m) < 0$ is equivalent to $g>g^*_d(m)$. 
	In addition, if $d=1$ and $V \geq 0$, then $g^*_d(m) = 0$, 
	while $g^*_d(m) > 0$ holds provided $d=2,3$. 
	Moreover, for all $m>0$ whenever $g<g^{*}_d(m)$, no minimizer for $\mathcal{E}_g$ with $\|u\|_2^2=m$ exists. 
\end{lemma}
\begin{proof}
	We first show that $e(g,m)<0$ for sufficiently large $g>0$. Let $r_0$ be given in \ref{A2}, and take $\varphi\in C_c^\infty(\mathbb{R}^d) \cap S_m$ 
	with $\supp \varphi\subset B_{r_0/2}(0)$. Since $x-y\in  B_{r_0}(0)$ for all $x,y \in  B_{r_0/2}(0)$, we have 
	\begin{equation*}
		e(g,m) \leq \mathcal{K}(\varphi)-\frac{g}{4}\iint_{ B_{r_0/2}(0)\times  B_{r_0/2}(0)} \vab{\varphi(x)}^2 V(x-y) \vab{\varphi(y)}^2 dx dy\leq K(\varphi)-\frac{g}{4} m^2 V_0,
	\end{equation*}
	where we use \ref{A2}: $\inf_{r\in  B_{r_0}(0)}V=V_0>0$. Clearly, one can make the above expression strictly negative 
	provided $g$ is sufficiently large. 
	Thus there exists $g_0 > 0$ such that $e(g,m)<0$ for every $g>g_0$, 
	and by continuity and the monotonicity in $g$, 
	we conclude the existence of $g^*_d(m)\in [0,\infty)$ such that $e(g,m)=0$ for all $g\leq g_d^*(m)$ and $e(g,m)<0$ for all $g> g_d^*(m)$.

	Our next aim is to prove $g^*_d(m) = 0$ when $d=1$ and $V \geq 0$. 
	Let $g > 0$ be arbitrary and fix $\varphi \in C^\infty_c(\R)$ so that $\varphi(-x) = \varphi(x) \geq 0$, $\varphi' \leq 0$ in $(0,\infty)$ and $\Vab{\varphi}_2^2 = m$. 
	Notice that $\varphi(0) > 0$ holds.  
	For $t \in (0,1)$, consider $\varphi_t(x) := t^{1/2} \varphi(t x)$. Then for any $g>0$, 
	\[
	e(g,m) \leq \cE_g(\varphi_t) = \frac{t^2}{2} \Vab{\varphi'}_2^2 - \frac{g}{4} t^2 \iint_{\R \times \R} \varphi^2 (t x) V(x-y) \varphi^2 (t y) dxdy.
	\]
	Since $V \geq 0$ and $x \varphi'(x) \geq 0$, the monotone convergence theorem and \ref{A2} imply 
	\[
	\lim_{t \to 0^+} \iint_{\R \times \R} \varphi^2 (t x) V(x-y) \varphi^2 (t y) dxdy = \varphi^4(0) \iint_{\R \times \R} V(x-y) dx dy = \infty. 
	\]
	Therefore, for a sufficiently small $t>0$, $e(g,m) \leq \cE_g(\varphi_t) < 0$ holds. 
	Since $g>0$ is arbitrary, we have $g_d^*(m) = 0$.

Next, when $d=2,3$, we prove that $g^*_d(m) > 0$ holds. 
By $V \in L^{d/2} (\R^d)$, as in the proof of \cref{prop1} (\eqref{bdd-2ndterm} by replacing $V_{\e,2}$ with $V$), 
there exists $C>0$ such that 
\begin{equation*}\label{SOB}
\mathcal{V}(u)\leq Cm \mathcal{K}(u) \quad \text{for all $u \in S_m$}. 
\end{equation*}
Hence, if $g<\frac{1}{Cm}$, then $\cE_g(u) \geq 0$ for all $u\in S_m$, and the fact $e(g,m) \leq 0$ yields $e(g,m)=0$. 
Thus, $e(g,m)=0$ for all sufficiently small $g$.

It remains to show that  for all $m>0$ and $g \in (0, g^{*}_d(m) )$, no minimizer for $\mathcal{E}_g$ with $\Vab{u}_2^2=m$ exists. 
To this end, it suffices to prove that if a minimizer for $\mathcal{E}_g$ with $g>0$ and $\Vab{u}_2^2=m$ exists, 
then $e(g',m)<0$ holds for all $g'>g$. Suppose that there is a function $u^g\in S_m$ satisfying $e(g,m)=\mathcal{E}_g(u^g)=\mathcal{K}(u^g)-g\mathcal{V}(u^g)$. 
By $\mathcal{K}(u^g)>0$ and $e(g,m)\leq0$, we have $\mathcal{V}(u^g)>0$, and hence 
$e(g',m)\leq \mathcal{K}(u^g)-g'\mathcal{V}(u^g)=e(g,m)+(g-g')\mathcal{V}(u^g)<0$. 
\end{proof}

We prove two last auxiliary results concerning the properties of $e(g,m)$ when $g$ is fixed and $m$ varies. 

\begin{lemma}\label{sad}
	Suppose {\rm \ref{A1}}. Then for all $g>0$, 
	the function $m\mapsto e(g,m)$ is non-increasing and subadditive, i.e. $e(g,m_1+m_2)\leq e(g,m_1)+e(g,m_2)$ for all $m_1, m_2>0$.
\end{lemma}
\begin{proof}
Note that if the subadditivity holds, then $m\mapsto e(g,m)$ is indeed non-increasing due to $e(g,m)\leq 0$. 
To prove the subadditivity, choose $\varphi_1, \varphi_2\in C_c^\infty(\mathbb{R}^d)$ with $\Vab{\varphi_1}_2^2=m_1, \Vab{\varphi_2}_2^2=m_2$ 
and consider, for $t>0$, $\psi_t \coloneq \varphi_1+\varphi_2(\cdot+te)$, where $e$ denotes a fixed unit vector in $\mathbb{R}^d$. For all sufficiently large $t$, the supports of $\varphi_1$ and $\varphi_2(\cdot+te)$ are disjoint, so that $\Vab{\psi_t}^2_2=m_1+m_2$ 
and $\cK(\psi_t)=\cK(\varphi_1)+\cK(\varphi_2)$. 
On the other hand, since $\varphi_i \in C^\infty_c(\R^d)$, there exists $R_t>0$ such that $R_t \to \infty$ as $t \to \infty$ and 
\[
\begin{aligned}
	&\iint_{\R^d \times \R^d} \vab{\varphi_1(x)}^2 V \ab( \pm \ab(  x-y + t e ) ) \vab{\varphi_2(y)}^2 dx dy  
	\\
	= \ & 
	\iint_{\R^d \times \R^d} \vab{\varphi_1(x)}^2 \ab( \chi_{ \R^d \setminus B_{R_t} (0) } V)  \ab( \pm \ab(   x-y+te ) ) \vab{\varphi_2(y)}^2 dx dy.
\end{aligned}
\]
Since $\| \chi_{\R^d \setminus B_{R_t}(0)} V \|_{ 1 } \to 0$ ($d=1$) and $\| \chi_{\R^d \setminus B_{R_t}(0)} V \|_{ d/2 } \to 0$ ($d=2,3$) as $t \to \infty$, 
\eqref{basineqcV} leads to 
\[
\iint_{\R^d \times \R^d} \vab{\varphi_1(x)}^2 V \ab( \pm  \ab( x-y + t e ) ) \vab{\varphi_2(y)}^2 dx dy  \to 0 \quad \text{as $t \to \infty$},
\]
which implies that as $t \to \infty$
\begin{equation*}
	\begin{aligned}
		&\mathcal{V}(\psi_t)-\mathcal{V}(\varphi_1)-\mathcal{V}(\varphi_2)
		\\
		= \ & 
		\iint_{\R^d \times \R^d} 
		\vab{\varphi_1(x)}^2  
		\Bab{ \vab{V(x-y+te)} + \vab{ V \ab( - (x-y+te ) ) } }  
		\vab{\varphi_2(y)}^2 dx dy \to 0. 
	\end{aligned}
\end{equation*}
Thus, for each $\varphi_1 \in C^\infty_c(\R^d) \cap S_{m_1}$ and $\varphi_2 \in C^\infty_c(\R^d) \cap S_{m_2}$,
\[
e(g,m_1+m_2) \leq \lim_{t \to \infty} \cE_g(\psi_t) = \cE_g(\varphi_1) + \cE_g(\varphi_2),
\]
which gives $e(g,m_1+m_2)\leq e(g,m_1)+e(g,m_2)$ as claimed. 
\end{proof}

\begin{lemma}\label{sad2}
	Assume {\rm \ref{A1}} and {\rm \ref{A2}}, and let $m>0$ and  $g>g_d^*(m)$. 
	Then for any $t>1$, $e(g,tm)< t^2 e(g,m)$ holds. 
\end{lemma}

\begin{proof}
Let $\lbrace u_n\rbrace \subset S_m$ be a minimizing sequence for $\mathcal{E}_g$. Then for any $t>1$ we have 
\begin{equation*}
	\begin{aligned}
		e(g,tm)\leq \mathcal{E}_g(\sqrt{t}u_n) 
		=t\mathcal{K}(u_n)-gt^2 \mathcal{V}(u_n) 
		&= t^2\cE_g(u_n) + t(1-t)\mathcal{K}(u_n)
		\\
		&=t^2 e(g,m) + t(1-t) \cK(u_n) + c_n,
	\end{aligned}
\end{equation*}
where $c_n \to 0$ as $n \to \infty$. 
Note that $\lim_n \mathcal{K}(u_n)>0$ holds; otherwise, $\Vab{u_n}_q \to 0$ for any $q \in (2,2^*)$, 
$\cV(u_n) \to 0$ in view of \eqref{vanicV} and the following contradiction happens: 
$0 > e(g,m) = \lim_{n \to \infty} \cE_g(u_n) \geq 0$. 
Thus, in the limit $n\rightarrow 0$ we obtain $e(g,tm)< t^2 e(g,m) $ for any $t>1$.
\end{proof}

\section{Existence of minimizers for $g>g_d^*(m)$}
\label{sec3}

\begin{theorem}\label{T:Exming>g^*}
	Assume {\rm \ref{A1}} and {\rm \ref{A2}}, and let $m>0$ and $g>g_d^*(m)$. 
	Then up to translations and subsequences, 
	any minimizing sequence for $e(g,m)$ is relatively compact in $H^1(\R^d)$. 
	In particular, there exists $u \in S_m$ such that $\cE_g(u) = e(g,m) < 0$.  
\end{theorem}

The proof makes use of the concentration compactness method (\cite{Cazenave1,Lions1,Lions2}). 
Choose $g>g_d^*(m)$. In the following, we shall denote by $\lbrace u_n\rbrace \subset S_m $ any minimizing sequence for $\mathcal{E}_g$, 
which is bounded in $H^1(\R^d)$ thanks to \cref{prop1}. 
\begin{lemma}\label{lions}
Let $Q=[0,1]^d$. Then 
\begin{equation}\label{lmm}
\liminf_{n\rightarrow \infty} \sup_{z\in \mathbb{Z}^d} \|u_n\|_{L^2(z+Q)}>0.
\end{equation}
\end{lemma}
\begin{proof}
Suppose on the contrary that $\liminf_{n\rightarrow \infty} \sup_{z\in \mathbb{Z}^d} \|u_n\|_{L^2(z+Q)}=0$, 
in particular there is a subsequence  $\lbrace u_{n_k} \rbrace$ with $\lim_{k\rightarrow \infty} \sup_{z\in \mathbb{Z}^d} \|u_{n_k}\|_{L^2(z+Q)}=0$. 
Since  $\lbrace u_n\rbrace $ is bounded in $H^1(\R^d)$, we have by Lions' lemma (\cite[Lemma I.1]{Lions2} and \cite[Lemma 1.21]{Willem}) 
\begin{equation*}
\|u_{n_k}\|_q\rightarrow 0 \quad \text{for all } q\in (2,2^*). 
\end{equation*}
Hence, \eqref{vanicV} gives $\cV(u_{n_k}) \to 0$ and a contradiction is obtained as in the proof \cref{sad2}. 
Therefore, \eqref{lmm} holds. 
%
%
%
\end{proof}

Now we prove \cref{T:Exming>g^*}:

\begin{proof}[Proof of \cref{T:Exming>g^*}]

By \cref{lions} there exists a sequence $\lbrace x_n \rbrace$ such that $\|u_n(\cdot+x_n)\|_{L^2(Q)}\rightarrow C_0>0$.
By the translation invariance of $\mathcal{E}_g$, 
we may suppose 
\begin{equation*}
 \|u_n\|_{L^2(Q)}\rightarrow C_0>0, \quad 
 u_n \rightharpoonup \omega \quad \text{weakly in $H^1(\R^d)$}, \quad 
 u_n \to \omega \quad \text{strongly in $L^2_{loc} (\R^d)$}. 
 \end{equation*}
Note that $\omega \not \equiv 0$: this follows from $\|u_n\|_{L^2(Q)}\rightarrow C_0>0$ and the Rellich--Kondrashov theorem.
Our next goal is to prove that $\|\omega\|_2^2=m$.

To prove $\omega \in S_m$ by contradiction, assume 
\begin{equation}\label{m2g0}
	0< m_1 \coloneq \Vab{\omega}_2^2 < m. 
\end{equation}
Set $m_2 \coloneq m - m_1 > 0$ and $v_n \coloneq u_n - \omega$. Then we claim 
\begin{equation}\label{e12}
	e(g,m)=\liminf_{n\rightarrow\infty}\mathcal{E}_g(u_{n})
	= 
	\mathcal{E}_g(\omega)+\liminf_{n \to \infty}\mathcal{E}_g(w_n). 
\end{equation}
In fact, from $v_n \rightharpoonup 0$ weakly in $H^1(\R^d)$, it follows that 
\[
\Vab{\nabla u_n}_2^2 = \Vab{\nabla (\omega + v_n)}_2^2 = \Vab{\nabla \omega}_2^2 + \Vab{\nabla v_n}_2^2 + o(1). 
\]
Furthermore, since $u_n = \omega + v_n$ and $\omega v_n \to 0$ strongly in $L^p(\R^d)$ with $1 \leq p < 2^*/2$, we have 
\[
\iint_{\R^d \times \R^d} \vab{ \omega (x) v_n (x) } \vab{V(x-y)} u_n^2(y) + u_n^2(x) \vab{V(x-y) \vab{ \omega(y) v_n(y) }} dxdy \to 0,
\]
which yields 
\begin{equation*}\label{deccV}
	\cV(u_n) = \cV(v_n) + \cV(w_n) + o_n(1). 
\end{equation*}
Therefore, \eqref{e12} holds.

From the Brezis--Lieb Lemma, 
\begin{equation*}
	\|v_n\|_2^2\rightarrow m_2= m-m_1. 
\end{equation*}
Set 
\begin{equation*}
\tilde{v}_n \coloneq \frac{\sqrt{m_2}}{\|v_n\|_2}v_n.
\end{equation*}
It it easily seen from the boundedness of $\lbrace u_n \rbrace$ that 
\begin{equation*}
\mathcal{E}_g( \omega )\geq e(g,m_1), \quad 
 \lim_{n \to \infty} \mathcal{E}_g(v_n)=\lim_{n\to\infty} \mathcal{E}_g(\tilde{v}_n)\geq e(g,m_2).
\end{equation*}
By \cref{sad} and \eqref{e12}, this actually gives 
\begin{equation}\label{eg_sum}
	e(g,m)=e(g,m_1)+e(g,m_2)
\end{equation}
and 
\begin{equation*}\label{limsvw}
	\mathcal{E}_g(\omega)= e(g,m_1), \quad  \lim_{n\rightarrow \infty} \mathcal{E}_g(\tilde{v}_n)= e(g,m_2).
\end{equation*}
Note that $e(g,m)<0$ implies $e(g,m_1)<0$ and $e(g,m_2)<0$; 
in fact, if, for instance, $e(g,m_2)=0$ then we must have $e(g,m)=e(g,m_1) < 0$, 
which contradicts \cref{sad2}, and similarly for the case $e(g,m_1)=0$. 
From $e(g,m_1)<0$ and $e(g,m_2) < 0$, up to a subsequence, 
\begin{equation}\label{cvw}
\mathcal{V}(\omega) >0 \quad \lim_{n\to\infty} \mathcal{V}(\tilde{v}_n)=C_v>0.
\end{equation}
We clearly have $e(g,m)\leq \mathcal{E}_g( \sqrt{m/m_1} \, \omega)$, $e(g,m)\leq \mathcal{E}_g( \sqrt{m/m_2} \, \tilde{v}_n)$ and thus
\begin{equation*}
	\begin{aligned}
		e(g,m) &=
		\frac{m_1}{m}e(g,m)+\frac{m_2}{m}e(g,m) 
		\\
		&\leq 
		\mathcal{E}_g(\omega)+ \mathcal{E}_g(\tilde{v}_n) 
		-g \ab(\frac{m}{m_1}-1) \mathcal{V}(\omega) -g \ab(\frac{m}{m_2}-1)\mathcal{V}(\tilde{v}_n). 
	\end{aligned}
\end{equation*}
Due to \eqref{eg_sum} and \eqref{cvw}, 
taking the limit $n\rightarrow \infty$ yields
\begin{equation*}
e(g,m)\leq e(g,m_1)+e(g,m_2)-\ab(\frac{m}{m_1}-1)\ \cV(\omega) -\ab(\frac{m}{m_2}-1)C_v<e(g,m),
\end{equation*}
which is a contradiction, rooted in the assumption \eqref{m2g0}. Hence, the weak limit of $u_n$, $\omega$, satisfies 
\begin{equation}\label{norm_L}
\|\omega\|_2^2=m=\|u_n\|_2^2.
\end{equation}
In particular, by \eqref{norm_L} and $u_n \rightharpoonup \omega$ weakly in $H^1(\R^d)$, 
$u_n \to \omega$ strongly in $L^2(\R^d)$. 
By the interpolation, $u_n \to \omega$ strongly in $L^q(\R^d)$ for each $q \in [2,2^\ast)$ 
and Young's inequality with $V \in L^{ \max \{ 1 , d/2 \} } (\R^d)$ implies 
\begin{equation}\label{V_L}
	\lim_{n\to\infty}\mathcal{V}(u_n)=\mathcal{V}(\omega).
\end{equation}
Finally, from $u_n \rightharpoonup \omega$ weakly in $H^1(\R^d)$ and \eqref{V_L}, we infer that 
\[
e(g,m) \leq \cE_g(\omega) \leq \liminf_{n\to\infty} \cE_g(u_n) = \lim_{n \to \infty} \cE_g(u_n) = e(g,m),
\]
which gives 
\begin{equation*}
	\cE_g(\omega) = e(g,m), \quad 
\lim_{n\to\infty}\mathcal{K}(u_n) = \mathcal{K}(\omega).
\end{equation*}
Hence, $u_n \to \omega$ strongly in $H^1(\R^d)$ and $\omega$ is a minimizer. 
\end{proof}

\section{Non--local Pokhozaev Identity}

In this section, we prove that any minimizer for $e(g,m)$ with $g>g^*_d(m)$ satisfies a version of the Pokhozaev identity well--known from the local theory \cite{BerestyckiLions} . 
Moreover, for later use, we also establish the Pokhozaev identity for critical points of $\cE_g |_{S_m}$ when $d=3$. 
The precise aim of this section is to prove the following:

\begin{proposition}\label{p:Poho}
	Suppose {\rm \ref{A1}}, {\rm \ref{A2}}, {\rm \ref{B1}}, $V(-x) = V(x)$ and $W(x) \coloneq x \cdot \nabla V(x) \in L^1(\R^d) + L^\infty(\R^d)$. 
	\begin{enumerate}
		\item 
		Suppose that $u \in S_m$ satisfies $\cE_g(u) = e(g,m) \leq 0$. 
		Then the following Pokhozaev identity holds: 
		\begin{equation}\label{Poho}
			\frac{d-2}{2} \Vab{\nabla u}_2^2 + \frac{d}{2} \lambda \Vab{u}_2^2 
			= 
			\frac{g}{4} \int_{\R^d} u^2 (W * u^2) dx + \frac{d}{2} g \int_{\R^d} u^2 (V* u^2) dx,
		\end{equation}
		where $\lambda \in \R$ is the corresponding Lagrange multiplier, that is, $u$ is a solution of 
		\[
		- \Delta u + \lambda u = g(V*u^2) u \quad \text{in} \ \R^d. 
		\]
		
		\item 
		If $d=3$, $V \geq 0$ and $u$ is a critical point of $\cE_g |_{S_m}$ with $u \geq 0$ (resp. $u \leq 0$) or $u(x) = u(\vab{x})$, 
		then $u$ satisfies \eqref{Poho}. 
	\end{enumerate}
\end{proposition}

To prove \cref{p:Poho}, we closely follow the arguments of \cite[section 3]{MorozVanShaftingen13}. 
For this purpose, decay estimates on $u$ and $\nabla u$ are necessary:

\begin{lemma}\label{l:decay}
	Assume {\rm \ref{A1}}, {\rm \ref{A2}} and $V(-x) = V(x)$. 
	\begin{enumerate}
		\item 
		If $u \in S_m$ satisfies $\cE_g(u) = e(g,m) \leq 0$, 
		then $u \in C^2(\R^d)$, the Lagrange multiplier $\lambda$ is strictly positive 
		and $u, \partial_i u$ display the following exponential decay: 
		there exists $C>0$ such that 
		\[
		\vab{u(x)} + \vab{\partial_i u (x)} \leq C e^{ - \sqrt{\lambda} \vab{x} / 2 } 
		\quad \text{for all $x \in \R^d$}. 
		\]
		
		\item 
		If $d=3$, $V \geq 0$ and $u$ is a critical point of $\cE_g|_{S_m}$ with $u \geq 0$ (resp. $u \leq 0$) or $u(x) = u(\vab{x})$, 
		then $u \in C^2(\R^d)$, $u(x) \to 0$ as $\vab{x} \to \infty$ 
		and $\vab{x} \vab{\partial_i u(x)} \in L^\infty(\R^d)$. 
	\end{enumerate}
\end{lemma}

\begin{proof}
We first notice that in each case, there exists a Lagrange multiplier $\lambda \in \R$ such that $u$ is a weak solution of 
\begin{equation}\label{EL-eq}
-\Delta u + \lambda u = g (V \ast u^2) u \quad \text{in} \ \R^d. 
\end{equation}
When $u \in S_m$ satisfies $\cE_g(u) = e(g,m) \leq 0$, $\lambda > 0$ holds. In fact, from \eqref{EL-eq} and $\cE_g(u) = e(g,m) \leq 0$ it follows that 
\[
\lambda m = - \Vab{\nabla u}_2^2 + g \int_{\R^d} (V\ast u^2) u^2 dx 
= - 4 \cE_g (u) + \Vab{\nabla u}_2^2 = - e(g,m) + \Vab{\nabla u}_2^2 > 0. 
\]

We next show $u \in C^2(\R^d)$ and $u(x) \to 0$ as $\vab{x} \to \infty$. 
For this purpose, write $c_u(x) \coloneq g (V \ast u^2)(x)$. 
By $V \in L^{ \max \{ 1 , d/2 \} } (\R^d)$ and $u^2 \in L^q(\R^d)$ for all $q \in [1,\infty)$ if $d=1,2$ 
and for all $q \in [1,3]$ if $d=3$, 
Young's inequality implies $c_u \in L^q(\R^d)$ for any $q \in [ \max \{1 , d/2\} ,\infty)$. 
Since $u$ is a weak solution of $-\Delta u +(\lambda - c_u) u = 0$ in $\R^d$, 
from Moser's iteration (\cite[Chapter 8]{GilbargTrudinger} and \cite[Chapter 4]{HanLin}), 
there exists $C>0$ such that
\[
\Vab{u}_{L^\infty(B_{1}(z))} \leq C \Vab{u}_{L^2(B_2(z))} \quad \text{for all $z \in \R^d$},
\]
which implies $u \in L^\infty(\R^d)$ and $u(x) \to 0$ as $\vab{x} \to \infty$. 
In particular, $u \in L^q(\R^d)$ for every $q \in [2,\infty]$ and $c_u = g (V*u^2) \in L^q(\R^d)$ for each $q \in [ \max \{  1, d/2\} ,\infty]$. 
From the $L^2$-theory, $u \in H^2_{loc} (\R^d)$ holds and by rewriting \eqref{EL-eq} as 
\[
-\Delta u = c_u u - \lambda u \in L^\infty(\R^d), 
\]
the $L^p$-theory (\cite[Chapter 9]{GilbargTrudinger}) implies that 
for any $q \in (1,\infty)$ there exists $C_q$ such that
\begin{equation}\label{W2qest}
\Vab{u}_{W^{2,q} (B_1(z)) } \leq C_q \ab( \Vab{u}_{L^q(B_2(z))} + \Vab{c_u u - \lambda u}_{L^q(B_2(z))} ) \quad \text{for any $z \in \R^d$}.
\end{equation}
Thus, $u \in C^1(\R^d)$ with $\nabla u \in L^\infty(\R^d)$. 
Hence, $c_u = g (V *u^2) \in C^1(\R^d)$ and $c_u(x) \to 0$ as $\vab{x} \to \infty$ due to $V \in L^{ \max\{1,d/2\} } (\R^d)$.  
By the Schauder theory, $u \in C^2(\R^d)$ holds.

Now we prove that $\vab{u}$ and $\vab{\partial_i u}$ decay exponentially when $\cE_g(u) = e(g,m) \leq 0$.  
In this case, we note that $v \coloneq \vab{u}$ is also a minimizer 
since $\Vab{\nabla v}_2 \leq \Vab{\nabla u}_2$ and $\cE_g(v) \leq \cE_g(u)$. 
To show the exponential decay of $\vab{u}$, it suffices to show it under $u \geq 0$. 
We first prove the exponential decay of $\vab{u}$. Let us suppose $u \geq 0$ and 
as shown in the above, $\lambda > 0$ holds in this case. 
Thanks to the facts that $\lambda > 0$ and $\vab{c_u(x)} + u(x) \to 0$ as $\vab{x} \to \infty$, there exists an $R_0>0$ such that 
\[
\lambda - c_u(x) \geq \frac{\lambda}{2} \quad \text{for each $\vab{x} \geq R_0$}.
\]
Hence, 
\[
0 = -\Delta u + \ab( \lambda - c_u ) u \geq -\Delta u + \frac{\lambda}{2} u \quad \text{for any $\vab{x} \geq R_0$}.
\]
Let $C_0>0$ be a constant satisfying 
\begin{equation*}
	\ab( \max_{|z|\leq R_0} u_0(z) ) \exp\ab(\sqrt{\frac{\lambda_0}{2}}R_0)\leq C_0.
\end{equation*}
Thus $u_0(x)\leq w_{R_0}(x) \coloneq C_0 \exp\left(-\sqrt{\frac{\lambda_0}{2}}|x|\right)$ for $|x|\leq R_0$, 
and it remains to show that the same bound holds also for $|x|>R_0$. Note that 
\begin{equation*}
	-\Delta w_{R_0} +\frac{\lambda_0}{2}w_{R_0}\geq 0 \quad \text{for} \  \vab{x}\geq R_0.
\end{equation*}
Hence, 
\[
-\Delta ( w_{R_0} - u  ) + \frac{\lambda}{2} (w_{R_0} -u) \geq 0 \quad 
\text{for every $\vab{x} \geq R_0$}. 
\]
Since $w_{R_0} > 0$ and $w_{R_0} (x) - u(x) \to 0$ as $\vab{x} \to \infty$, 
the function $w_{R_0} - u$ cannot take a strictly negative minimum in $\vab{x} > R_0$, which implies 
$w_{R_0} - u \geq 0$ in $\R^d \setminus B_{R_0}(0)$. 
Thus, the desired exponential decay holds for $\vab{u}$.

The exponential decay estimate of $\vab{\partial_i u}$ follows from \eqref{W2qest} and 
$\Vab{u}_{C^{1}(B_1(z))} \leq C_q \Vab{u}_{W^{2,q} (B_1(z))}$ for each $q>d$ where 
$C_q$ is independent of $z \in \R^d$.

Finally, we prove $\vab{x} \vab{\partial_i u(x)} \in L^\infty(\R^3) $ for any critical point $u$ of $\cE_g|_{S_m}$ 
when $d=3$, $V \geq 0$ and either $u \geq 0$ (resp. $u \leq 0$) or $u(x) = u(\vab{x})$. 
If $\lambda > 0$ in \eqref{EL-eq}, then $\vab{\nabla u}$ decays exponentially and the desired property holds. 
Therefore, it is enough to treat the case $\lambda \leq 0$.

Suppose $u \geq 0$. By $V \geq 0$, $c_u = g(V*u^2) \geq 0$. From $\lambda \leq 0$ and $u \geq 0$, it follows that 
\[
-\Delta u = c_u u - \lambda u \geq 0 \quad \text{in} \ \R^3. 
\]
But there exists no non--negative twice differentiable $L^2$ function $u$  with $u\nequiv 0$ and  $-\Delta u\geq 0$ \cite[Lemma A.2]{ikoma} (think of the case $d=1$ as a motivation). 
Since $u \in L^2(\R^3)$ and $\Vab{u}_2^2 = m$, we obtain a contradiction. 
Thus, when $u \geq 0$, the Lagrange multiplier is always positive and 
the desired decay estimate holds when $u \geq 0$. 
The case where $u \leq 0$ is reduced to the former case by considering $-u$.

On the other hand, let $u(x) = u(\vab{x})$. It is known that $\vab{u(x)} \leq C \Vab{u}_{H^1} \vab{x}^{-1} $ for each $\vab{x} \geq 1$ 
(see \cite[Radial Lemma A.II.]{BerestyckiLions}). From this estimate, it follows that for any $q \in (1,\infty)$ and $z \in \R^3$ with $\vab{z} \geq 4$, 
\[
\Vab{u}_{L^q(B_2(z))} + \Vab{c_u u - \lambda u}_{L^q(B_2(z))} \leq \frac{C_q}{\vab{z}}. 
\]
By \eqref{W2qest} with $q>d$, we get $\vab{z} \vab{\partial_i u(z)} \leq C_q$ for every $\vab{z} \geq 4$ and 
this completes the proof for radial $u$. 
\end{proof}

\begin{remark}\label{positivity}
If $\lambda > 0$ and $u \geq 0$ satisfies 
\[
-\Delta u + \lambda u = c_u u \quad \text{in} \ \R^d, \quad c_u(x) \coloneq g (V*u^2) (x),
\]
then by rewriting the equation as $-\Delta u + (\lambda + c_u^-) u = c_u^+ u \geq 0$ 
where $c^\pm_u \coloneq \max \Set{ \pm c_u , 0 } \geq 0$, 
the strong maximum principle yields $u > 0$ in $\R^d$. 
In particular, if $u \geq 0$ is a minimizer of $e(g,m) \leq 0$, then $u > 0$ in $\R^d$. 
\end{remark}

We now prove \cref{p:Poho}:

\begin{proof}[Proof of \cref{p:Poho}]
We prove assertions 1 and 2 simultaneously. 
Pick a radially symmetric decreasing function $\varphi_1\in C_c^{\infty}(\mathbb{R}^d)$ with the properties
\begin{equation*}
	0\leq \varphi_1 \leq 1 \text{ in }  \mathbb{R}^d, \quad  \varphi_1\equiv 1 \text{ on } B_1(0),  \quad \varphi_1\equiv 0 \text{ on } \mathbb{R}^d\backslash B_2(0)
\end{equation*}
and set, for $n\geq 1$, 
\begin{equation*}
	\varphi_n(x):=\varphi_1(nx).
\end{equation*}
Note that for all $w\in L^1(\mathbb{R}^2)$, 
\begin{equation}\label{lims}
	\begin{aligned}
		&\int_{\mathbb{R}^d} \vab{\varphi_n w} dx \to \|w\|_1 \text{ as } n\to \infty, \\
		&\int_{\mathbb{R}^d} \vab{x} \vab{ \nabla \varphi_n (x)} \vab{w (x)}  d x \leq \ab(\max_{1\leq |y|\leq 2}|y| |\nabla \varphi_1(y)|) 
		\int_{n\leq |x|\leq 2n} |w| dx \to 0 \text{ as } n\to \infty. \\
	\end{aligned}
\end{equation}
Let $u$ be as in 1 or 2 in \cref{p:Poho} and set 
\begin{equation*}
	v_n(x):=\varphi_n(x) x \cdot \nabla u(x) \in H^1(\mathbb{R}^d).
\end{equation*}
Since $u$ is a solution to
\[
-\Delta u + \lambda u = g (V*u^2) u \quad \text{in} \ \R^d
\]
for some $\lambda \in \R$, we multiply the equation by $v_n$ and integrate over $\mathbb{R}^d$:
\begin{equation}\label{pkz_n}
	\int_{\R^d} \nabla u \cdot \nabla v_n+\lambda u v_n dx = g\int_{\R^d} (V\ast u^2) u v_n dx.  
\end{equation}
We have, by the integration by parts and a straightforward computation, 
\begin{equation*}
	\int_{\R^d} \nabla u \cdot \nabla v_n dx 
	=\int_{\R^d} \ab ( \nabla u \cdot \nabla \varphi_n x \cdot \nabla u + \frac{2-d}{2} \varphi_n \vab{\nabla u}^2 - \frac{1}{2} x \cdot \nabla \varphi_n \vab{\nabla u}^2  )  
	dx .
\end{equation*}
It follows from \eqref{lims} that 
\begin{equation*}
	\int_{\R^d} \vab{ \nabla u \cdot \nabla \varphi_n x \cdot \nabla u } + \vab{x \cdot \nabla \varphi_n \vab{\nabla u}^2} dx
	\leq 
	2\int_{\R^d} \vab{x} \vab{\nabla \varphi_n} \vab{\nabla u}^2 dx \to 0,
\end{equation*}
which implies 
\begin{equation}\label{1st-t}
	\lim_{n\to \infty} \int_{\R^d} \nabla u \cdot \nabla v_n dx = \frac{2-d}{2} \Vab{\nabla u}_2^2.  
\end{equation}
Integrating by parts, we have in a similar fashion 
\begin{equation}\label{middle}
	\lambda \int_{\R^d} u_0 v_n dx = -\lambda	\int_{\R^d} ( x \cdot \nabla \varphi_n + d \varphi_n ) \frac{|u|^2}{2}d x 
	\to -\frac{d \lambda}{2}\|u\|_2^2.
\end{equation}
We are left with the interaction term. 
From $V \in C^1(\R^d)$ and $\varphi_k u^2 \in C^1_c(\R^d)$ for every $k \in \N$, 
it follows that 
\[
V * (\varphi_k  u^2 ) \in C^1(\R^d), \quad \nabla (V \ast \varphi_k u^2) = (\nabla V) \ast (\varphi_k u^2). 
\]
Notice that 
\[
\int_{\R^d} (V * u^2) u v_n dx = \lim_{k \to \infty} \int_{\R^d} \ab( V * (\varphi_k u^2 ) ) uv_n dx, 
\]
and that the integrating by parts gives 
\begin{equation*}
	\begin{aligned}
		\int_{\mathbb{R}^d} \ab( V* (\varphi_k u^2) ) u v_n dx 
		&= 
		\int_{\mathbb{R}^d} \ab( V* (\varphi_k u^2) )  \varphi_n x_i \partial_i \ab(\frac{u^2}{2}) dx 
		\\
		&=
		-\frac{1}{2}\int_{\mathbb{R}^d } x_i \ab (\partial_{x_i} V * (\varphi_k u^2)  ) \varphi_n u^2 dx\\
		&\quad - \frac{1}{2} \int_{\R^d} \ab( V * (\varphi_k u^2) ) (x \cdot \nabla \varphi_n(x)) u^2 dx \\
		&\quad - \frac{d}{2}\int_{\mathbb{R}^d } \ab( V * (\varphi_k u^2) ) \varphi_n u^2 dx. 
	\end{aligned}
\end{equation*}
By \eqref{lims}, it is easily seen that 
\[
\lim_{n \to \infty} \lim_{k \to \infty} \int_{\R^d} \ab( V * (\varphi_k u^2) ) (x \cdot \nabla \varphi_n(x)) u^2 dx 
= \lim_{n \to \infty} \int_{\R^d} (V*u^2) (x \cdot \nabla \varphi_n(x)) u^2 dx = 0
\]
and 
\[
\lim_{n \to \infty} \lim_{k \to \infty}- \frac{d}{2}\int_{\mathbb{R}^d } \ab( V * (\varphi_k u^2) ) \varphi_n u^2 dx 
= - \frac{d}{2} \int_{\R^d} (V * u^2) u^2 dx. 
\]
Therefore, 
\[
\begin{aligned}
	&\lim_{n \to \infty} \int_{\R^d} (V *u^2) u v_n dx 
	\\
	=\ & 
	\lim_{n \to \infty} \lim_{k \to \infty} -\frac{1}{2} \int_{\R^d} x_i \ab( \partial_{x_i} V * (\varphi_ku^2) ) \varphi_n u^2 dx 
	- \frac{d}{2} \int_{\R^d} (V*u^2) u^2 dx .
\end{aligned}
\]
The first term on the right hand side is rewritten as 
\begin{equation*}
	\begin{aligned}
		&-\frac{1}{2}\iint_{\mathbb{R}^d\times \mathbb{R}^d} x_i (\partial_{x_i} V )(x-y)\varphi_n(x) u^2(x) \varphi_k(y) u^2(y) dx dy\\
		&=-\frac{1}{2}\iint_{\mathbb{R}^d\times \mathbb{R}^d} W(x-y) u^2(x) u^2(y) \varphi_k(y)\varphi_n(x) dx dy\\
		&\quad -\frac{1}{2}\int_{\mathbb{R}^d} dx \int_{\mathbb{R}^d} dy y_i (\partial_{y_i} V )(x-y)\varphi_n(x) u^2(x) u^2(y) \varphi_k(y) dx dy.
	\end{aligned}
\end{equation*}
Since $W\in L^1(\R^d) + L^\infty(\R^d)$ and $u^2 \in L^r(\R^d)$ for each $r \in [1,\infty]$, we have 
\[
\begin{aligned}
	&\lim_{n \to \infty}\lim_{k \to \infty} -\frac{1}{2}\iint_{\mathbb{R}^d\times \mathbb{R}^d} W(x-y) u^2(x) u^2(y) \varphi_k(y)\varphi_n(x) dx dy 
	\\
	= & - \frac{1}{2} \iint_{\R^d \times \R^d} W(x-y) u^2(x) u^2(y) dx dy. 
\end{aligned}
\]
On the other hand, the integration by parts leads to 
\begin{equation*}
	\begin{aligned}
		& - \frac{1}{2} \int_{\mathbb{R}^d}  y_i (\partial_{y_i} V )(x-y) u^2(y) \varphi_k(y) dy
		\\
		= & \ 
		\frac{1}{2} \int_{\R^d} y_i \partial_{y_i} \ab( V(x-y) ) u^2(y) \varphi_k(y) dy
		\\
		 = & - \frac{d}{2} \int_{\R^d} V(x-y)u^2(y) \varphi_k(y) dy-\int_{\R^d} V(x-y)u(y)y \cdot \nabla u(y) \varphi_k(y) dy
		 \\
		 & - \int_{\R^d} V(x-y) u^2 (y) y \cdot \nabla \varphi_k (y) dy. 
	\end{aligned}
\end{equation*}
From $w(y) \coloneq y \cdot \nabla u(y) \in L^\infty(\R^d)$ due to \cref{l:decay}, Young's inequality yields 
$(V*(wu) ) u^2 \in L^1(\R^d)$. Therefore, 
\[
\begin{aligned}
	&\lim_{n \to \infty}\lim_{k \to \infty} -\frac{1}{2}\int_{\mathbb{R}^d} dx \int_{\mathbb{R}^d} dy y_i (\partial_{y_i} V )(x-y)\varphi_n(x) u^2(x) u^2(y) \varphi_k(y)
	\\
	= \ & - \frac{d}{2} \int_{\R^d} (V * u^2) u^2 dx - \int_{\R^d} \ab( V \ast w u ) u^2 dx
\end{aligned}
\]
and $V(-x) = V(x)$ gives 
\[
\lim_{n \to \infty} \int_{\R^d} (V *u^2) u v_n dx = \int_{\R^d} (V*u^2) uw dx = \int_{\R^d} u^2 (V*uw) dx. 
\]
Thus, 
\begin{equation*}
	\begin{aligned}
		\int_{\R^d } u^2 (V *uw) dx 
		&= \lim_{n \to \infty} \int_{\R^d} (V*u^2) w v_n dx
		\\
		& = 
		- \frac{d}{2} \int_{\R^d} (V*u^2) u^2 dx 
		- \frac{1}{2} \int_{\R^d} (W*u^2) u^2 dx - \frac{d}{2} \int_{\R^d} (V*u^2) u^2 dx
		\\
		&\quad  - \int_{\R^d} (V*wu) u^2 dx,
	\end{aligned}
\end{equation*}
which yields 
\begin{equation}\label{last-term}
	\lim_{n\to\infty} \int_{\R^d} (V*u^2) u v_n dx =  \int_{\R^d} u^2 (V*uw) dx 
	= -\frac{1}{4} \int_{\R^d} (W*u^2) u^2 dx - \frac{d}{2} \int_{\R^d} (V*u^2) u^2 dx. 
\end{equation}
Combining \eqref{pkz_n}, \eqref{1st-t}, \eqref{middle} and \eqref{last-term}, we have the claimed identity
\begin{equation*}
	\frac{2-d}{2} \Vab{\nabla u}_2^2 - \frac{d\lambda}{2} \|u\|^2_2 
	= 
	-g\ab(\frac{1}{4}\int_{\R^d} (W\ast u^2)u^2 dx + \frac{d}{2}\int_{\R^d} (V\ast u^2)u^2 dx ). 
\end{equation*}
This completes the proof. 
\end{proof}

\section{Proof of \cref{thm:gstar}}

\subsection{Existence of minimizers for $g=g_3^*(m)$}

Here we focus on $d=3$ and prove that the transition is of first order via
\begin{proposition}\label{pt_2}
Let $d=3$, $m>0$ and $g=g_3^*(m)$. Then there exists a $u \in S_m$ such that $\cE_{g_3^*} (u) = e(g_3^*,m) = 0$. 
\end{proposition}
To show that the trapping transition is of first order in $d=3$, we need to prove the existence of a minimizer for $g=g_3^*(m)$, where $e(g_3^*(m),m)=0$. Note that the former proof does not apply as it relied on having $e(g,m)<0$ as in the proof of \cref{lions}. 
We shall overcome this with the connection of the problem to a linear Schrödinger problem, as we shall explain. 

Consider any sequence $\lbrace g_j \rbrace_j$ with $g_j\downarrow g_3^*(m)$. 
By \cref{T:Exming>g^*}, we can find a corresponding sequence $u_j$ of minimizers of $\mathcal{E}_{g_j}$ with $\|u_j\|_2^2=m$, 
and, by \cref{prop1} and \cref{cont_mon}, it forms a minimizing sequence for $\mathcal{E}_{g_3^*(m)}$ and 
$\lbrace u_j \rbrace$ is bounded in $H^1(\R^3)$. 
For each $u_j$, consider the Schrödinger operator 
\begin{equation*}
H_j \coloneq -\frac{1}{2}\Delta-\frac{g_j}{2} V * u^2_j ,
\end{equation*}
with $\Delta$ denoting the $3d$ free laplacian. It is easy to verify, using similar computations as those leading to \cref{prop1}, 
that $H_j$ is well--defined via the Friedrich's extension of the quadratic form in $L^2(\mathbb{R}^3)$
\begin{equation*}
\mathcal{S}_j(\psi) \coloneq \frac{1}{2}\int_{\R^3} |\nabla\psi|^2 dx -\frac{g_j}{2} \int_{\R^3} |\psi|^2  V * u_j^2  dx.
\end{equation*} 
It is well--known that if 
\begin{equation*}
e_j \coloneq \inf_{\|\psi\|_2=1}\mathcal{S}_j(\psi)<0,
\end{equation*}
then $e_j$ is an eigenvalue of $H_j$ (see, for instance, \cite[11.5 THEOREM]{LiebLoss} or \cite[Lemma 8.1.11]{Cazenave1}). 
Since $\cE_{g_j} (u_j) = e(g_j,m) < 0$, it follows that $\cV (u_j) > 0$ and consequently, for all $g_j$ we have 
\begin{equation}\label{lt11}
e_j\leq \mathcal{S}_j\ab( \frac{u_j }{\sqrt{m}})= \frac{1}{m} \cE_{g_j} (u_j) - \frac{g_j}{4m} \int_{\R^3} u_j^2 (V*u_j^2) dx < \frac{1}{m}e(g_j,m)<0.
\end{equation}
Hence for each $j$, $H_j$ has at least one negative eigenvalue. 

By the Lieb--Thirring (or the Cwikel–Lieb–Rozenblum) inequality \cite[Theorem 4.31]{FrankLaptevWeidl}, 
the number of negative eigenvalues of a Schrödinger operator $-\Delta-\tilde{V}$ on $L^2(\mathbb{R}^d)$, $N_e$, is bounded as 
\begin{equation*}
N_e\leq C \int_{\mathbb{R}^d} \ab( \tilde{V}_+)^{d/2} dx
\end{equation*}
provided $d\geq 3$. Hence, by \eqref{lt11} and $g_j \to g_3^*(m) > 0$, 
there is a constant $C>0$, which is independent of $j$, such that for all $j \geq 1$, 
\begin{equation}\label{LT}
 1\leq C \int_{\R^3} \ab( V\ast u_j^2 )_+^{3/2} dx \leq C\int_{\R^3} \pab{V_+  \ast u_j^2}^{3/2} dx.
\end{equation}
This bound enables us to argue as in the proof of \cref{lions} in the critical case. 
\begin{lemma}\label{lions_crit}
	Let $d=3$ and $g_j > g^*_3(m)$ satisfy $g_j \to g_3^*(m)$. 
	For each $j$, choose a minimizer $u_j$ of $\mathcal{E}_{g_j}$ with $\|u_j\|_2^2=m$. Let $Q=[0,1]^d$. Then there is a subsequence such that 
\begin{equation}\label{lmm2}
\lim_{j\rightarrow \infty} \sup_{z\in \mathbb{Z}^d} \|u_j\|_{L^2(z+Q)}>0.
\end{equation}
\end{lemma}

\begin{proof}
Suppose that \eqref{lmm2} does not hold. Then, by Lions' Lemma, after taking a subsequence (still denoted by $\lbrace u_j \rbrace$), 
\begin{equation}\label{lions2}
	\lim_{j\to\infty}\|u_{j}\|_q=0.
\end{equation}
for all $q\in (2,6)$. Let us choose $R>0$ and write 
\begin{equation*}
	V^R_+:=\chi_{B_R(0)} V_+, \quad V^{R,c}_+:=V_+-V^R_+.
\end{equation*}
Since $V_+\in L^{3/2}(\mathbb{R}^3)$, $\lim_{R\to \infty}\|V_+^{R,c}\|_{3/2}=0$ and $V_+^R \in L^1(\mathbb{R}^3)$ for any $R>0$.
By Young's inequality,
\begin{equation*}
	\begin{aligned}
		&\|V^R_+\ast |u_j|^2\|_{3/2}\leq \|V_+^R\|_1 \|u_j\|_3^2,\\
		&\|V^{R,c}_+\ast |u_j|^2\|_{3/2}\leq \|V_+^{R,c}\|_{3/2} \|u_j\|_2^2=\|V_+^{R,c}\|_{3/2}m\\
	\end{aligned}
\end{equation*}
and thus \eqref{lions2} leads to 
\begin{equation*}\label{bd_vplus}
	\limsup_{j \to \infty}\int_{\R^3} (V_+\ast |u_j|^2)^{3/2} dx 
	\leq  
	\limsup_{j \to \infty} \Bab{\|V_+^R\|_1 \|u_j\|_3^2+\|V_+^{R,c}\|_{3/2}m} 
	= \| V_+^{R,c} \|_{3/2} m.
\end{equation*}
By letting $R \to \infty$, we obtain 
\[
\int_{\R^3} (V_+ * u_j^2 )^{3/2} dx \to 0.
\]
However, this and \eqref{LT} give the following contradiction:
\[
1 \leq C \limsup_{j \to \infty} \int_{\R^3} \pab{  V_+ * u_j^2}^{3/2} dx = 0.
\]
Therefore \eqref{lmm2} holds.
\end{proof}

Now we complete the proof of \cref{pt_2}. 
We proceed as in the proof of \cref{T:Exming>g^*}: up to translations and subsequences, 
\cref{lions_crit} and the Rellich--Kondrashov theorem ensure that $u_j \rightharpoonup \omega_c \not \equiv 0$ weakly in $H^1(\R^3)$. 
Assume that $\|\omega_c\|_2^2<m$. 
With $v_j \coloneq u_j - \omega_c$, 
a similar argument to the proof of \cref{T:Exming>g^*} yields (see \eqref{e12}) 
\begin{equation}\label{sum_k}
0=e(g^*_3(m), m) = \lim_{j \to \infty} \cE_{g_j} (u_j) \geq  \mathcal{E}_{ g^*_3(m) }(\omega_c)+ \liminf_{j \to \infty} \mathcal{E}_{g_j}(v_j).
\end{equation}
This and \cref{sad} imply 
\begin{equation}\label{lims_of_vw}
\begin{aligned}
&\mathcal{E}_{g_3^*(m)} (\omega_c) 
\geq e( g_3^*(m) , \Vab{\omega_c}_2^2  ) \geq e ( g_3^*(m) , m  ) = 0,  \\
& \liminf_{j \to \infty} \mathcal{E}_{g_j}(v_j)\geq e(g_3^*(m),m-\|\omega_c\|_2^2) \geq e ( g_3^*(m) , m  ) =  0.
\end{aligned}
\end{equation}
Thus, \eqref{sum_k} and \eqref{lims_of_vw} yield
\begin{equation*}
\lim_{j\to \infty} \mathcal{E}_{g_j}(v_j)=\mathcal{E}_{g_3^*(m)}(\omega_c)=0.
\end{equation*}
By $\omega_c \not \equiv 0$ and $\cK(\omega_c) >0$ with $\cE_{g^*_3(m)}(\omega_c) = 0$, 
$\cV(\omega_c) > 0$ follows. 
Since we have assumed $\|\omega_c\|_2^2<m$, for $\tau \coloneq \sqrt{m} / \|\omega_c\|_2>1$, 
\begin{equation*}
\begin{aligned}
0=e(g_3^*(m),m)\leq &\mathcal{E}_{g_3^*(m)}(\tau \omega_c)=\tau^2 \mathcal{E}_{g_3^*(m)}(\omega_c)+(\tau^2-\tau^4)g_3^*(m)\mathcal{V}(\omega_c) < 0,
\end{aligned}
\end{equation*}
which is a contradiction. Hence, $\|\omega_c\|_2^2=m$. 
From $u_j\to \omega$ weakly in $H^1(\mathbb{R}^3)$ and $\Vab{u_j}_2^2 \to \Vab{\omega_c}_2^2$, 
we have $\|u_j-\omega_c\|_2\to0$ and by \eqref{sum_k} and \eqref{lims_of_vw},
\begin{equation*}
0=\lim_{j\to \infty} \cE_{g_j}(u_j)\geq \liminf_{j\to \infty} \mathcal{E}_{g_j}(v_j)\geq \mathcal{E}_{g_3^*(m)}(\omega_c)\geq 0
\end{equation*}
and $\omega_c$ is a minimizer. This completes the proof of \cref{pt_2}.

\subsection{Non--existence of minimizers for $g=g_2^*(m)$}

We complete the proof of \cref{thm:gstar} by showing that the transition is of second order in $d=2$. 
\begin{proposition}\label{pt_3}
Let $d=2$, $m>0$ and $g=g_2^*(m)$. Assume {\rm \ref{A1}}, {\rm \ref{A2}} and {\rm \ref{B1}--\ref{B4}}. 
Then there is no minimizer for $e(g^*_2(m), m)=0$. 
\end{proposition}

To prove \cref{pt_3}, we first prove the following simple result: 
\begin{proposition}\label{radial}
	Let $d \in \N$,  $f \colon \mathbb{R}^d\rightarrow \mathbb{R}$ and $g \colon \mathbb{R}^d\rightarrow \mathbb{R}$ be non-negative, 
	Lebesgue measurable, vanishing at infinity and radial. 
	Assume that $f=f^*$ and $g=g^*$, where $^*$ denotes the symmetric-decreasing rearrangement, and that  
	\begin{equation*}
		h(x)=\int_{\R^d} f(x-y)g(y) d y
	\end{equation*} 
	is well-defined. Then $h$ is radial, non-negative and non-increasing, i.e., $h=h^*$. 
\end{proposition}

\begin{proof}
The	nonnegativity of $h$ is clear. To prove $h=h^*$, define
	\begin{equation}\label{ball}
		v_{r,r'}(x) \coloneq \mathrm{Vol}\left\lbrace B_r(x)\cap B_{r'}(0)\right\rbrace
	\end{equation} 
	where $B_r(y)$ denotes the open ball centered at $y\in \mathbb{R}^d$ of radius $r$. 
	The function $v_{r,r'}$  is clearly radial and  non-increasing in $|x|$ for all $r,r'$. By the layer cake representation 
	\cite[1.13 THEOREM]{LiebLoss}
	and Fubini's Theorem, we have 
	\begin{equation*}
		h(x)=\int_0^\infty dt \int_0^\infty ds \int_{\R^d} dy \chi_{ \lbrace f>t \rbrace}(x-y)\chi_{ \lbrace g>s \rbrace}(y). 
	\end{equation*}
	Since $f=f^*$ and $g=g^*$, 
	level sets $\lbrace f>t \rbrace$ and $\lbrace g>s \rbrace$ are open balls located at the origin 
	(see \cite[Chapter 3]{LiebLoss}). 
	Hence 
	\begin{equation*}
		\int_{\R^d} \chi_{\lbrace f>t \rbrace }(x-y)\chi_{\lbrace g>s \rbrace }(y) d y
	\end{equation*} 
	is, for all $t, s$ in $\mathbb{R}_{\geq 0}$, of the form \eqref{ball} for some $r, r'$. 
	Since $v_{r,r'}$ is radial and non-increasing, so is $h$ and this completes the proof. 
\end{proof}

Now we give a proof of \cref{pt_3}:

\begin{proof}[Proof of \cref{pt_3}]
We argue by contradiction and suppose that $u_0 \in S_m$ is a minimizer for $e(g^*_2(m), m)$. 
Remark that by replacing $u_0$ to $\vab{u_0}$ we may assume $u_0 \geq 0$ without loss of generality. 
Furthermore, \ref{A1} and \ref{B2} imply $V=V^*$. Since $u_0^* \in S_m$, 
the strict Riesz rearrangement inequality (\cite[3.7 THEOREM]{LiebLoss}) yields $\cE_{g^*_2(m)} (u_0^*) \leq \cE_{g^*_2(m)} (u_0)$. 
Thus, we also may suppose that $u_0$ is radially symmetric and non-increasing. Recall also 
the Euler--Lagrange equation: 
\[
-\Delta u_0 + \lambda_0 u_0 = g^*_2(m) (V * u_0^2) u_0 \quad \text{in} \ \R^2. 
\]

Let us denote 
\begin{equation*}
F(x):=\frac{1}{2}V(x)+\frac{1}{4}W(x)
\end{equation*} 
where $W(x)=x \cdot \nabla V(x)$. 
From \ref{A1}, \ref{B1} and \ref{B3}, integration by parts gives 
\begin{equation*}
\int_{\mathbb{R}^2} F dx =0. 
\end{equation*}
By \ref{B4} and $V(0) = \Vab{V}_\infty$ due to \ref{B2}, $F(x)\geq 0$ for $|x|\leq r^*$ and $F(x)<0$ for $|x|>r^*$.
Moreover, when $g=g_2^*(m)$, then $e(g_2^*(m),m)=0$ and hence 
\begin{equation*}
	\|\nabla u_0\|_2^2=\frac{g_2^*(m)}{2} \int_{\R^2} (V*u_0^2) u_0^2 dx, \quad 
	\lambda_0 \Vab{u_0}_2^2 = \frac{g_2^*(m)}{2} \int_{\R^2} (V*u^2_0) u_0^2 dx .
\end{equation*}
The Pokhozaev identity \eqref{Poho} hence reduces in this case to
\begin{equation}\label{pokh_2}
	\int_{\mathbb{R}^2}(F\ast u_0^2)u_0^2 dx =0.
\end{equation}
Notice that Fubini's theorem implies 
\begin{equation*}
\int_{\mathbb{R}^2} (F\ast u_0^2)u_0^2dx =\int_{\mathbb{R}^2} F( u_0^2\ast u_0^2)dx, 
\end{equation*}
Define $h \coloneq u_0^2 \ast u_0^2$. 
By \cref{radial}, $h$ is radial and non--increasing, and thus 
\begin{equation}\label{infs}
\inf_{|x|<r^*} h = \sup_{|x|>r^*} h.
\end{equation}
We also remark that $h>0$ since $u_0>0$ thanks to \cref{positivity}. 
Moreover, we have already observed that $\int_{\mathbb{R}^2} F dx =0$, and hence
\begin{equation}\label{ints}
\int_{|x|<r^*} F dx =\int_{\vab{x}>r^*} |F| dx .  
\end{equation}
Thus, we can estimate 
\begin{equation}\label{est}
\begin{aligned}
\int_{\R^2} F h dx &= \int_{|x|< r^*} F h dx -\int_{|x|> r^*} |F| h dx 
\\
&\geq \ab (\inf_{|x|< r^*} h )\int_{|x|<r^*}F dx  -\int_{|x|>r^*} |F| h dx
\\\
& > \ab (\inf_{|x|<r^*} h ) \int_{|x|<r^*}F dx  - \ab(\sup_{|x|>r^*} h ) \int_{|x|>r^*} |F| dx 
\\
&= \ab (\inf_{|x|<r^*} h - \sup_{|x|>r^*} h) \int_{\vab{x}<r^*} F dx = 0,
\end{aligned}
\end{equation}
where we used \eqref{ints} and then \eqref{infs}. 
Note that the second inequality in \eqref{est} is strict since $0<h$ in $\R^2$ and $h(x) \to 0$ as $\vab{x} \to \infty$. 
Thus, $\int_{\R^2} Fh dx >0$ while on the other hand $\int_{\R^2} Fh dx=0$ by \eqref{pokh_2}. This is a contradiction 
and any minimizer for $\mathcal{E}_{g_2^*(m)}$ with $\|u\|_2^2=m$ cannot exist. This proves \cref{pt_3}.
\end{proof}

\cref{thm:gstar} now follows from \cref{pt_2,pt_3}.

\section{Proof of \cref{thm:2ndsol-nonex}}

\subsection{Nonexistence of critical points}

This subsection is devoted to proving the third assertion in \cref{thm:2ndsol-nonex}, namely, 
the nonexistence of any critical point $u$ of $\cE_{g}|_{S_m}$ with $\pm u \geq 0$ or $u(x) = u(\vab{x})$ when $d=3$ and $g >0$ is small enough. 

\begin{proof}[Proof of \cref{thm:2ndsol-nonex},  assertion 2]
Let $g > 0$ and $u \in S_m$ be any critical point of $\cE_g|_{S_m}$ with $\pm u \geq 0$ or $u(x) = u(\vab{x})$. 
Then there exists the Lagrange multiplier $\lambda \in \R$ such that 
\[
-\Delta u + \lambda u = g (V * u^2) u \quad \text{in} \ \R^3
\]
and according to \cref{p:Poho} $u$ satisfies the Pokhozaev identity: 
\[
\frac{1}{2} \Vab{\nabla u}_2^2 + \frac{3}{2} \lambda \Vab{u}_2^2 = \frac{g}{4} \int_{\R^3} (W * u^2) u^2 dx + \frac{3}{2} g \int_{\R^3} (V*u^2)u^2 dx.
\]
From $\Vab{\nabla u}_2^2 + \lambda \Vab{u}_2^2 = g \int_{\R^3} (V*u^2)u^2 dx$ it follows that 
\[
\Vab{\nabla u}_2^2 = - \frac{g}{4} \int_{\R^3} (W*u^2)u^2 dx. 
\]
Young's and Sobolev's inequality give 
\[
\vab{\int_{\R^3} (W*u^2) u^2 dx } 
\leq \Vab{u}_{2}^2 \Vab{W}_{3/2} \Vab{u}_6^2 \leq C_S \Vab{W}_{3/2} m \Vab{\nabla u}_2^2,
\]
which implies 
\[
\Vab{\nabla u}_2^2 \leq \frac{C_S m g}{4} \Vab{W}_{3/2} \Vab{\nabla u}_2^2. 
\]
Since $g$ and $u$ are arbitrary, if 
\[
0 < g < g_2 \coloneq \frac{4}{C_Sm \Vab{W}_{3/2} },
\]
then $\nabla u \equiv 0$ and $\cE_g|_{S_m}$ does not admit any critical point $u$ with $u \geq 0$ or $u(x) = u(\vab{x})$. 
This completes the proof. 
\end{proof}

\subsection{Existence of positive energy solution}

In this subsection, we prove the first 
in \cref{thm:2ndsol-nonex}.  
We proceed through the mountain pass theorem as in  \cite{Jeanjean,BMRV21,ikomatanaka,MoRiVe22}.
Throughout this subsection, the following conditions are always assumed:  
\begin{equation}\label{assump}
\text{$d=3$, $m>0$\ref{A1} and \ref{A2}}, V\quad \text{radial}. 
\end{equation}


Define 
\[
H^1_r (\R^3) \coloneq \Set{ u \in H^1(\R^3) | \text{$u$ is radial} }, \quad S_{m,r} \coloneq S_m \cap H^1_r(\R^3)
\]
and set 
\[
e_r(g,m) \coloneq \inf_{u \in S_{m,r}} \cE_g(u). 
\]

\begin{remark}\label{r:e=e_r}
	By definition, $e(g,m) \leq e_r(g,m)$ holds. 
	When $V$ is non-increasing, by the symmetric-decreasing rearrangement and the Riesz rearrangement inequality, 
	$\cE_g(u^*) \leq \cE_g(u)$ holds for any $u \in S_m$, which yields $e_r(g,m) = e(g,m)$ for each $g>0$ and $m>0$. 
\end{remark}

Notice that under \eqref{assump}, \cref{prop1} and \cref{cont_mon,l:gstar} with $e_r(g,m)$ hold true. 
Thus, there exists $ g^*_{3,r} =  g^*_{3,r} (m) > 0$ such that $e_r(g,m) < 0$ is equivalent to $g> g^*_{3,r} (m)$. 
Furthermore, \cref{sad2} also holds with $e_r(g,m)$. 
Then we may show that $e_r(g,m)$ is attained if $g\geq g^*_{3,r}$:

\begin{lemma}\label{l:radmini}
	Suppose \eqref{assump}. 
	If $g>g^*_{3,r} (m)$, then there exists $u \in S_{m,r}$ such that $\cE_g(u) = e_r(g,m) < 0$. 
	Furthermore, when $g=g^*_{3,r}$, there exists $v \in S_{m,r}$ such that $\cE_{g^*_{3,r}} (v) = e_r(g^*_{3,r} , m ) = 0$. 
\end{lemma}

\begin{proof}
Let $g>g^*_{3,r}(m)$ and $\lbrace u_n \rbrace \subset S_{m,r}$ be any minimizing sequence for $e_r(g,m)$. 
As pointed in the above, $\lbrace u_n \rbrace$ is bounded in $H^1_r(\R^3)$ and 
after taking a subsequence if necessary, we may assume $u_n \rightharpoonup u$ weakly in $H^1_r(\R^3)$. 
Since $H^1_r(\R^3) \subset L^q(\R^3)$ is compact for every $q \in (2,6)$, 
Young's inequality with $V \in L^{3/2} (\R^3)$ gives $\cV(u_n) \to \cV(u)$, which yields 
\[
0> e_r(g,m) = \liminf_{n \to \infty} \cE_g(u_n) \geq \cE_g(u) \geq e_r \ab (g , \Vab{u}_2^2). 
\]
\cref{sad2} with $e_r(g,m)$ leads to $\Vab{u}_2^2 = m$ and $u$ is a minimizer.

For the case $g=g^*_{3,r}(m)$, we argue as in the proof of \cref{thm:gstar}. 
Let $\lbrace g_j \rbrace_j$ satisfy $g_j>g^*_{3,r}(m)$ and $g_j \to g^*_{3,r}(m)$. Choose $u_j \in S_{m,r}$ so that $\cE_{g_j} (u_j) = e_r(g_j,m) < 0$. 
Then $\lbrace u_j \rbrace$ is bounded in $H^1_r(\R^3)$ and the Lieb-Thirring (the Cwikel–Lieb–Rozenblum) inequality implies that 
there exists $C>0$ such that for all $j \geq 1$, 
\[
1 \leq C \int_{\R^3} (V_+ * u_j^2)^{3/2} dx. 
\]
Let $u_j \rightharpoonup \omega_c$ weakly in $H^1_r(\R^3)$. 
If $\omega_c \equiv 0$, then by $\Vab{u_j}_q \to 0$ for $q \in (2,6)$, 
the argument in \cref{lions_crit} is valid and a contradiction occurs. 
Thus, $\omega_c \not \equiv 0$ and $g_j \cV (u_j) \to g^*_{3,r} (m) \cV(\omega_c)$. 
Hence, 
\[
0 = e_r(g^*_{3,r} (m) , m) = \lim_{j \to \infty} \cE_{g_j} (u_j) \geq \cE_{g^*_{3,r} (m)} (\omega_c) \geq e( g^*_{3,r} (m) , \Vab{\omega_c}_2^2 ) = 0.
\]
Then the argument same to the proof of \cref{thm:gstar} works in this case and we conclude $\Vab{\omega_c}_2^2 = m$. 
Therefore, $\omega_c$ is the desired minimizer. 
\end{proof}

\begin{remark}\label{r:g3*=g3*r}
	When $V$ is non-increasing, 
	since $e(g,m) = e_r(g,m)$ holds by \cref{r:e=e_r}, 
	we have $g^*_3(m) = g^*_{3,r}(m)$.
\end{remark}

To prove that $\cE_g|_{S_m}$ admits the mountain pass geometry, 
we need the following result: 
\begin{lemma}\label{l:posbdry}
	Let $\wt{g} > 0$ be given and \eqref{assump} hold. 
	Then there exists $\rho_0=\rho_0(\wt{g},m)>0$ such that
	\[
	\inf \Set{ \cE_g(u) | u \in S_{m,r}, \ \Vab{\nabla u}_2^2 = \rho^2 } \geq \frac{\rho^2}{8} 
	\quad \text{for each $g \in (0,\wt{g}]$ and $\rho \in (0,\rho_0]$}. 
	\]
	As a consequence, if $g \in (0,\wt{g}]$ and $u \in S_{m,r}$ satisfies $\cE_g(u) \leq 0$, then $\Vab{\nabla u}_2 > \rho_0$ holds. 
\end{lemma}

\begin{proof}
For $R > 1$, write $V_{R,1} \coloneq \chi_{B_R(0)} V$ and $V_{R,2} \coloneq V - V_{R,1}$. 
Thanks to $V \in L^{3/2} (\R^3)$, $\Vab{V_{R,2}}_{3/2} \to 0$ as $R \to \infty$. 
Young's and Sobolev's inequality imply that for all $u \in S_{m,r}$ 
\[
\iint_{\R^3\times \R^3} \vab{V_{R,2}(x-y)} u^2(x) u^2(y) dx dy 
\leq \Vab{u}_2^2 \Vab{V_{R,2}}_{3/2} \Vab{u}_6^2 \leq C_S m \Vab{V_{R,2}}_{3/2} \Vab{\nabla u}_2^2. 
\]
On the other hand, by $V_{R,1} \in L^1(\R^3)$ and the Gagliardo--Nirenberg inequality, 
for each $u \in S_{m,r}$, 
\[
\int_{\R^3 \times \R^3} \vab{V_{R,1}(x-y)} u^2(x) u^2(y) dx dy 
\leq \Vab{V_{R,1}}_1 \Vab{u}_4^4 \leq C \Vab{V_{R,1}}_1 \Vab{\nabla u}_2^3 \sqrt{m}. 
\]
Therefore, for any $u \in S_{m,r}$ and $g \in (0,\wt{g}]$, 
\[
\begin{aligned}
	\cE_g(u) &= 
	\frac{1}{2} \Vab{\nabla u}_2^2 - \frac{g}{4} \int_{\R^3\times \R^3} \ab( V_{R,1} (x-y) + V_{R,2} (x-y) ) u^2(x) u^2(y) dx dy
	\\
	&\geq \frac{1}{2} \Vab{\nabla u}_2^2 - g C m \Vab{V_{R,2}}_{3/2} \Vab{\nabla u}_2^2 - Cg \Vab{V_{R,1}}_1 \sqrt{m} \Vab{\nabla u}_2^3
	\\
	&= 
	\Vab{\nabla u}_2^2 \ab[ \frac{1}{2} - g Cm \Vab{V_{R,2}}_{3/2} - Cg \Vab{V_{R,1}}_1 \sqrt{m} \Vab{\nabla u}_2 ]. 
\end{aligned}
\]
Fix $R(g,m) = R >1$ so large that $\wt{g}Cm \Vab{V_{R,2}}_{3/2} \leq 1/4$ and choose $\rho_0 = \rho_0 (\wt{g},m) > 0$ such that 
$C \wt{g} \Vab{V_{R,1}} \sqrt{m} \rho_0 \leq 1/8$. 
Then for every $\rho \in (0,\rho_0]$, $g \in (0,\wt{g}]$ and $u \in S_{m,r}$ with $\Vab{\nabla u}_2 = \rho$, 
\[
\cE_g(u) \geq \frac{\rho^2}{8},
\]
which completes the proof. 
\end{proof}


In what follows, we assume the following condition:
\begin{equation}\label{ex:nonposel}
	\begin{dcases}
	&\text{there exist $\rho_1>0$ and $0 \leq u_1 \in S_{m,r}$ such that}\\
	& \Vab{\nabla u_1}_2 > \rho_1, \quad  \inf \Set{ \cE_g(u) | u \in S_{m,r}, \Vab{\nabla u}_2=\rho_1 } \geq \frac{\rho^2_1}{8} > \cE_g(u_1). 
	\end{dcases}
\end{equation}

\begin{remark}\label{range-g}
	When $g \geq g^*_{3,r}$, \eqref{ex:nonposel} holds due to \cref{l:posbdry} and the existence of minimizer in \cref{l:radmini}. 
	Furthermore, there exists $\delta_0>0$ such that \eqref{ex:nonposel} holds true for any $g \in [g^*_{3,r} -\delta_0 , g^*_{3,r}]$. 
	In fact, let $u_1$ be a minimizer for $e_r(g^*_{3,r}, m)$ and $\rho_0 > 0$ be the constant in \cref{l:posbdry} with $\wt{g} \coloneq g^*_{3,r}$. 
	By $\cE_{\wt{g}}(\vab{u_1}) = \cE_{\wt{g}}(u_1)=0$ and $\vab{u_1} \in S_{m,r}$, 
	without loss of generality, we may assume $u_1 \geq 0$. 
	Then \cref{l:posbdry} leads to $\Vab{\nabla u_1}_2 > \rho_0$ and 
	\[
	\inf \Set{ \cE_g(u) | u \in S_{m,r} , g \in (0, g^*_{3,r} ], \Vab{\nabla u}_2 = \rho_0 } \geq \frac{\rho_0^2}{8}. 
	\]
	Therefore, for a small $\delta_0>0$ and any $g \in [g_{3,r}^*-\delta_0, g_{3,r}^*]$, we obtain 
	\[
	\inf \Set{ \cE_g(u) | u \in S_{m,r} , \Vab{\nabla u}_2 = \rho_0 } \geq \frac{\rho_0^2}{8} > \cE_g(u_1),
	\]
	which implies that \eqref{ex:nonposel} holds for all $g \in [g^*_{3,r} - \delta_0 , g^*_{3,r} ]$.  
\end{remark}

To proceed, we define an auxiliary functional $\wt{\cE_g} \colon \R \times H^1_r(\R^3) \to \R$ based on the scaling 
(see \cite{Jeanjean,BMRV21,ikomatanaka,MoRiVe22}):
\[
\wt{\cE}_g (\theta, u) \coloneq \cE_g \ab( e^{3\theta/2} u(e^\theta \cdot) ) 
= \frac{e^{2\theta}}{2} \Vab{\nabla u}_2^2 - \frac{g}{4} \iint_{\R^3 \times \R^3} V \ab( \frac{x-y}{e^\theta} ) u^2(x) u^2(y) dx dy. 
\]
Under $W(x) = x \cdot \nabla V(x) \in L^1(\R^3) + L^\infty(\R^3)$, we can prove $\wt{\cE}_g \in C^1(\R \times H^1_r(\R^3), \R)$ and 
\[
\partial_\theta \wt{\cE}_g (\theta, u) = e^{2\theta} \Vab{\nabla u}_2^2 + \frac{g}{4} \iint_{\R^3 \times \R^3} W \ab( \frac{x-y}{e^\theta} ) u^2 (x) u^2(y) dx dy.
\]
Indeed, the formula can be checked for radial functions in $C^\infty_c(\R^3)$. 
For general $u \in H^1_r(\R^3)$, choose $\lbrace u_n \rbrace \subset C^\infty_c(\R^d) $ such that $u_n$ is radial and $\Vab{u_n - u}_{H^1} \to 0$. 
Since $\lbrace \cE_g (\cdot, u_n) \rbrace_n$ and $\lbrace \partial_\theta \cE_g (\cdot, u_n) \rbrace_n$ 
converge locally uniformly with respect to $\theta\in \R$, 
we see that $\wt{\cE}_g (\cdot, u)$ is of class $C^1$ for general $u \in H^1_r(\R^3)$. 
It is now not difficult to show $\wt{\cE}_g \in C^1(\R \times H^1_r(\R^3) , \R)$.

Under \eqref{ex:nonposel}, we define the mountain pass values of $\cE_g$ and $\wt{\cE}_g$: 
\begin{equation*}
	c_{mp} \coloneq  \inf_{\gamma \in \Gamma} \max_{0 \leq t \leq 1} \cE_g( \gamma(t) ),
	\quad 
	\wt{c}_{mp} \coloneq \inf_{\wt{\gamma} \in \wt{\Gamma}} \max_{0 \leq t \leq 1} \wt{\cE}_g \ab( \wt{\gamma}(t) ),
\end{equation*}
where 
\[
\begin{aligned}
	\Gamma &\coloneq \Set{ \gamma \in C([0,1] , S_{m,r} ) | \Vab{\nabla \gamma(0)}_2 < \rho_1, \gamma(1) = u_1 }, 
	\\
	\wt{\Gamma} &\coloneq \Set{ \wt{\gamma} = (\gamma_\theta , \gamma_u ) \in C( [0,1] , \R \times S_{m,r} ) | 
		\gamma_\theta(0) = 0 = \gamma_\theta (1), \ \gamma_u \in \Gamma }.
\end{aligned}
\]

\begin{lemma}\label{l:mp}
	Under \eqref{ex:nonposel}, $\Gamma, \wt{\Gamma} \neq \emptyset$ and $ \rho_1^2/8 \leq c_{mp} = \wt{c}_{mp}$. 
\end{lemma}

\begin{proof}
For $u_1$ in \eqref{ex:nonposel}, consider 
\[
\gamma_0(t) \coloneq \ab( \e(1-t) + t )^{3/2} u_1 \ab( \ab( \e(1-t) + t ) x )  \in C([0,1] , S_{m,r} ). 
\]
For a sufficiently small $\e>0$, observe that $\Vab{\nabla \gamma_0(0) }_2^2 = \e^2 \Vab{\nabla u_1}_2^2 < \rho_1^2$. 
Hence, $\gamma \in \Gamma$. Since $(0,\gamma) \in \wt{\Gamma}$ for any $\gamma \in \Gamma$,  
$\wt{\Gamma} \neq \emptyset$ holds, and $c_{mp}, \wt{c}_{mp}$ are well-defined.

For each $\gamma \in \Gamma$ there exists $s \in (0,1)$ such that $\Vab{\nabla \gamma (s)}_2 = \rho_1$. Thus, \eqref{ex:nonposel} leads to 
\[
\max_{0 \leq t \leq 1} \cE_g(\gamma(t)) \geq \cE_g(\gamma(s)) \geq \frac{\rho_1^2}{8}. 
\]
Hence we obtain  $\rho_1^2/8 \leq c_{mp}$.

To show $c_{mp} = \wt{c}_{mp}$, note first that clearly $\wt{c}_{mp} \leq c_{mp}$. 
On the other hand, for any $\wt{\gamma} = (\gamma_\theta, \gamma_u) \in \wt{\Gamma}$, 
a path defined by 
\[
\gamma_0 (t) (x) \coloneq e^{ 3 \gamma_\theta(t) /2 } \gamma_u(t) \ab( e^{ \gamma_\theta (t)} x )
\]
satisfies $\gamma_0 \in C([0,1] , S_{m,r})$, $\Vab{\nabla \gamma_0(0)}_2 = \Vab{\nabla \gamma_u(0)}_2 < \rho_1$ and 
$\gamma_0(1) = \gamma_u(1) = u_1$. Thus $\gamma_0 \in \Gamma$ and 
\[
\max_{0 \leq t \leq 1} \wt{\cE}_g \ab(  \wt{\gamma} (t) ) = \max_{0 \leq t \leq 1} \cE_g(\gamma_0(t)) \geq c_{mp}.
\]
Since $\wt{\gamma} \in \wt{\Gamma}$ is arbitrary, $c_{mp} \leq \wt{c}_{mp}$ holds and this completes the proof.
\end{proof}

Select $\lbrace \gamma_n \rbrace \subset \Gamma$ so that 
\[
\max_{0 \leq t \leq 1} \cE_g(\gamma_n(t)) \to c_{mp}. 
\]
Since $ \cE_g( \vab{\gamma_n(t)} ) \leq  \cE_g(\gamma_n(t))$ and the map $[0,1] \ni t \mapsto \vab{\gamma_n(t)} \in S_{m,r}$ 
belongs to $\Gamma$ due to $\Vab{ \nabla \vab{\gamma_n(0)} }_2 \leq  \Vab{\nabla \gamma_n(0)}_2$ and 
$\vab{\gamma_n(1)} = \vab{u_1} = u_1$, 
without loss of generality, we may suppose $\gamma_n(t) \geq 0$ for all $n \in \N$ and $t \in [0,1]$. 
Note that \cref{l:mp} implies that 
\[
\max_{0 \leq t \leq 1} \wt{\cE}_g (\wt{\gamma}_n(t)) \to \wt{c}_{mp} = c_{mp}, \quad \wt{\gamma}_n \coloneq (0, \gamma_n) \in \wt{\Gamma}. 
\]
By applying Ekeland's variational principle or the deformation lemma (for instance, see the arguments in \cite[Lemma 5.15 and Theorem 5.16]{Willem}), 
there exists $\lbrace (\theta_n,u_n) \rbrace$ such that 
\begin{equation}\label{propPS}
	\begin{dcases}
		&\wt{\cE}_g(\theta_n,u_n) \to c_{mp}, \quad 
		\Vab{d_{S_{m,r}} \wt{\cE}_{g} (\theta_n, u_n)}_{ T_{u_n}^* S_{m,r}  }  \to 0, 
		\\
		&\partial_\theta \wt{\cE}_g (\theta_n,u_n) \to 0, 
		\quad  \dist_{\R \times S_{m,r}} \ab( (\theta_n,u_n) , \wt{\gamma}_n([0,1]) ) \to 0,
	\end{dcases}
\end{equation}
where $d_{S_{m,r}} \wt{\cE}_g$ is the partial derivative of $\wt{\cE}_g |_{\R \times S_{m,r}}$ on $\R \times S_{m,r}$, 
$T_u S_{m,r}$ is the tangent space of $S_{m,r}$ at $u \in S_{m,r}$, that is, 
\[
T_u S_{m,r} = \Set{ w \in H^1_r(\R^d) | \int_{\R^d} u w dx = 0 }
\]
and 
\[
\Vab{ d_{S_{m,r}} \wt{\cE}_g (\theta,u)  }_{ T^*_u S_{m,r} } \coloneq \sup \Set{ \partial_u\wt{\cE}_g(\theta,u) w | w \in T_u S_{m,r}, \ \Vab{w}_{H^1} \leq 1 }.
\]

\begin{proposition}\label{ConvPS}
	Suppose \eqref{assump}, \eqref{ex:nonposel} and $V \geq 0$. 
	Then any sequence $\lbrace (\theta_n,u_n) \rbrace \subset \R \times S_{m,r}$ satisfying \eqref{propPS} contains 
	a strongly convergent subsequence in $\R \times H^1_r(\R^3)$. 
\end{proposition}

\begin{proof}
By \eqref{propPS}, $\theta_n \to 0$. This fact with $\wt{\cE}_g(\theta_n,u_n) = c_{mp} + o(1)$ and \cref{prop1} imply that 
$\lbrace u_n \rbrace$ is bounded in $H^1_r(\R^3)$. 
By taking a subsequence if necessary, 
we may assume $u_n \rightharpoonup u_\infty$ weakly in $H^1_r(\R^3)$ 
and $u_n \to u_\infty$ strongly in $L^q(\R^3)$ for any $q \in (2,6)$. 
Remark that $u_\infty \geq 0$ in $\R^3$ holds due to \eqref{propPS} and $\gamma_n(t) \geq 0$, and 
that $\cV(u_n) \to \cV(u_\infty)$ holds. 
Since $\| d_{S_{m,r}} \wt{\cE}_g(u_n) \|_{T_{u_n}^* S_{m,r} } \to 0$, there exists $\lbrace \lambda_n \rbrace \subset \R$ such that 
\begin{equation}\label{fullPS}
\Vab{\partial_u \wt{\cE}_g(\theta_n, u_n) + \lambda_n Q'(u_n) }_{(H^1_r(\R^3))^*} \to 0, \quad 
\text{where} \ Q(u) \coloneq \frac{1}{2} \int_{\R^3} \vab{u}^2 dx. 
\end{equation}
Since $\theta_n \to 0$ and $\{u_n\}$ is bounded in $H^1(\R^3)$, so is $\{\lambda_n\}$. Let $\lambda_n \to \lambda_\infty$. 
From 
\[
\Vab{ V - V \ab( \frac{\cdot}{e^{\theta_n}} ) }_{3/2} \to 0
\]
and Young's inequality, we infer that 
\begin{equation}\label{conv-inter}
V \ab( \frac{\cdot}{e^{\theta_n}} )\ast u_n^2 \to  V * u_\infty^2 \quad 
\text{strongly in $L^q(\R^3)$ for $q \in (3/2,\infty)$}. 
\end{equation}
Hence, $u_\infty$ satisfies 
\begin{equation*}
-\Delta u_\infty + \lambda_\infty u_\infty = g( V * u_\infty^2 ) u_\infty  \quad \text{in} \ \R^3. 
\end{equation*}

Our next aim is to prove $u_\infty \not \equiv 0$. 
To this end, assume on the contrary that $u_\infty \equiv 0$. 
Then $u_n \to 0$ strongly in $L^q(\R^3)$ for each $q \in (2,6)$ and 
$\cV(u_n) \to 0$, which leads to 
\[
\frac{e^{2\theta_n}}{2} \Vab{\nabla u_n}_2^2 = \wt{\cE}_g(\theta_n,u_n) + o(1) \to c_{mp} > 0.
\]
On the other hand, from 
\[
o(1) = \partial_\theta \wt{\cE}_g(\theta_n,u_n) 
= e^{2\theta_n} \Vab{\nabla u_n}_2^2 + \frac{g}{4} \iint_{\R^3 \times \R^3} W \ab( \frac{x-y}{e^{\theta_n}} ) u^2_n(x) u^2_n(y) dx dy,
\]
it follows that 
\begin{equation}\label{Wpos}
\frac{g}{4} \iint_{\R^3 \times \R^3} W \ab( \frac{x-y}{e^{\theta_n}} ) u^2_n(x) u^2_n(y) dx dy  \to -2 c_{mp} < 0. 
\end{equation}
By assumption, $W$ can be decomposed as $W=W_1+W_\infty$ 
where $W_1 \in L^1(\R^3)$ and $W_\infty \in L^\infty(\R^3)$ with $W_\infty(x) \to 0$ as $\vab{x} \to \infty$. 
Young's inequality and the fact $\theta_n \to 0$ imply  
\begin{equation}\label{W1}
\vab{\iint_{\R^3 \times \R^3} W_1 \ab( \frac{x-y}{e^{\theta_n}} ) u^2_n(x) u^2_n(y)  } 
\leq e^{3\theta_n} \Vab{W_1}_1 \Vab{u_n}_{4}^4 \to 0.
\end{equation}
For $W_\infty$, write $W_{\infty, R,1} \coloneq \chi_{B_R(0)} W_\infty$ and $W_{\infty,R,2} \coloneq W_\infty - W_{\infty,R,1}$. 
Then $W_{\infty,R,1} \in L^1(\R^3)$ and $\Vab{W_{\infty,R,2}}_\infty \to 0$ as $R \to \infty$. Thus, 
\[
\begin{aligned}
	&\limsup_{n \to \infty} \vab{\iint_{\R^3 \times \R^3} W_\infty \ab( \frac{x-y}{e^{\theta_n}} ) u^2_n(x) u^2_n(y)  }  
	\\
	\leq \ & 
	\limsup_{n \to \infty} e^{3\theta_n}\Vab{W_{\infty,R,1}}_1 \Vab{u_n}_4^4 
	+ \limsup_{n \to \infty} \Vab{W_{\infty,R,2}}_\infty \Vab{u_n}_2^4 
	= \Vab{W_{\infty,R,2}}_\infty m^2
\end{aligned}
\]
and letting $R \to \infty$ yields 
\begin{equation}\label{Winfinity}
\lim_{n \to \infty}  \vab{\iint_{\R^3 \times \R^3} W_\infty \ab( \frac{x-y}{e^{\theta_n}} ) u^2_n(x) u^2_n(y)  }  = 0.
\end{equation}
However, \eqref{W1} and \eqref{Winfinity} contradict \eqref{Wpos}. Therefore, $u_\infty \not \equiv 0$.

We next claim $\lambda_\infty > 0$. 
On the contrary, suppose $\lambda_\infty \leq 0$. 
It follows from $u_\infty \geq 0$ and $V \geq 0$ that 
\begin{equation}\label{limiteq}
-\Delta u_\infty = - \lambda_\infty u_\infty + g(V* u_\infty^2) u_\infty \geq 0 \quad \text{in} \ \R^3. 
\end{equation}
Then \cite[Lemma A.2]{ikoma} and $u_\infty \in L^2(\R^3)$ give $u_\infty \equiv 0$, which is a contradiction. 
Hence $\lambda_\infty > 0$.

Finally, we show $\Vab{u_n - u_\infty}_{H^1} \to 0$. 
It follows from \eqref{fullPS}, \eqref{conv-inter}, \eqref{limiteq} and $\theta_n \to 0$ that 
\[
\begin{aligned}
	\Vab{\nabla u_\infty}_2^2 + \lambda_\infty \Vab{u_\infty}_2^2 
	&\leq \liminf_{n \to \infty} \ab[ e^{2\theta_n} \Vab{\nabla u_n}_2^2 + \lambda_n \Vab{u_n}_2^2 ] 
	\\
	& \leq \limsup_{n \to \infty} \ab[ e^{2\theta_n} \Vab{\nabla u_n}_2^2 + \lambda_n \Vab{u_n}_2^2 ] 
	\\
	&= \limsup_{n \to \infty} g \int_{\R^3 \times \R^3} V \ab( \frac{x-y}{e^{\theta_n}} ) u_n^2(x) u_n^2(y) dxdy
	\\
	&=g \int_{\R^3} (V* u_\infty^2) u_\infty^2 dx = \Vab{\nabla u_\infty}_2^2 + \lambda_\infty \Vab{u_\infty}_2^2.
\end{aligned}
\]
These inequalities yield 
\[
\Vab{\nabla u_n}_2^2 + \lambda_\infty \Vab{u_n}_2^2 \to \Vab{\nabla u_\infty}_2^2 + \lambda_\infty \Vab{u_\infty}_2^2.
\]
By $\lambda_\infty > 0$ and $u_n \rightharpoonup u_\infty$ weakly in $H^1(\R^3)$, 
we have $u_n \to u_\infty$ strongly in $H^1(\R^3)$ and the desired property holds. 
\end{proof}

Now we complete the proof of the first assertion in \cref{thm:2ndsol-nonex}. 

\begin{proof}[Proof of \cref{thm:2ndsol-nonex} assertion 1]
By \cref{range-g}, there exists $g_1 \in (0, g^*_{3,r} )$ such that for all $g \geq g_1$, \eqref{ex:nonposel} holds. 
Hence, for each $g \geq g_1$, \cref{ConvPS} implies the existence of $\omega_g \in S_{m,r}$ such that 
$\cE_g(\omega_g) = c_{mp} > 0$ and $d_{S_{m,r}} \cE_g(\omega_g) = 0$. 
Moreover, if $V$ is non-increasing, then \cref{r:g3*=g3*r} gives $g^*_3(m) = g^*_{3,r}$ 
and this completes the proof of \cref{thm:2ndsol-nonex}. 
\end{proof}

\subsection{Existence of local minimizers}
We now pass to the proof of the second assertion in \cref{thm:2ndsol-nonex}. 
Recall that we are interested in the following minimization problem 
	\begin{equation}\label{loc-minpro-sec6}
		\wt{e} (g,m) \coloneq \inf \Set{ \cE_g (u) | u \in S_m, \ \Vab{\nabla u}_2^2 > \frac{\rho_0^2}{4}  }.
	\end{equation}
Below, 
we prove that every minimizing sequence to \eqref{loc-minpro-sec6} has 
a strongly covergent subsequence up to translations.

For the existence of $\rho_0>0$ 
a result analogous to \cref{l:posbdry} is necessary. Its proof is almost identical to that result, and we omit it.

\begin{lemma}\label{l:wall}
	For any given $\wt{g} > 0$ and $\wt{m}>0$, there exists $\rho_0=\rho_0(\wt{g},\wt{m})>0$ such that
	\[
	\inf \Set{ \cE_g(u) | u \in S_{m'}, \ \Vab{\nabla u}_2^2 = \rho^2 } \geq \frac{\rho^2}{8} 
	\quad \text{for each $g \in (0,\wt{g}]$, $m' \in (0,\wt{m}]$ and $\rho \in (0,\rho_0]$}. 
	\]
\end{lemma}

In what follows, we fix $m>0$, select $\rho_0 > 0$ in \cref{l:wall} with $\wt{g} \coloneq g^*_d(m)$ and $\wt{m} \coloneq m$ 
and consider the following minimization problem 
for each $g \in (0,g^*_d(m)]$: 
\begin{equation}\label{def:local_problem}
\wt{e} (g,m) = \inf_{ u \in \mathcal{A}_m } \cE_g(u), \quad 
\mathcal{A}_m \coloneq \Set{ u \in S_m | \Vab{\nabla u}_2^2 > \frac{\rho_0^2}{4} }.
\end{equation}
We next show 

\begin{lemma}\label{Lem:mono}
	The function $ (0,m] \ni m' \mapsto \wt{e} (g,m') \in \R$ is nonincreasing. 
\end{lemma}

\begin{proof}
Let $0<m_1 < m_2 \leq m$. For any $u \in C^\infty_c(\R^3) \cap \mathcal{A}_{m_1}$, 
pick any $\psi \in C^\infty_c(\R^3)$ with $\Vab{\psi}_2^2 = m_2 - m_1$ and $\supp \psi \subset \R^3 \setminus B(0,1)$, 
and consider 
$\psi_t (x) \coloneq t^{3/2} \psi( t x )$ for $t \in (0,1)$. Since 
\[
\cE_g(\psi_t) = \frac{t^2}{2} \Vab{\nabla \psi}_2^2 
- \frac{g}{4} \int_{\R^3} \ab( V_t * \psi^2 ) \psi^2 \odif{x},
\]
where $V_t(x) \coloneq V(x/t)$, 
by $\Vab{V_t}_{3/2} \to 0$ as $t \to 0^+$, 
we see that $\cE_g (\psi_t) \to 0$ as $t \to 0^+$. 
From $\supp u \cap \supp \psi_t = \emptyset$ for each $0< t \ll 1$, 
it follows that
\[
u + \psi_t \in S_{m_2}, \quad 
\frac{\rho_0^2}{4} < \Vab{\nabla u}_2^2 < \Vab{\nabla (u + \psi_t)}_2^2, \quad 
\wt{e} (g,m_2) \leq \lim_{t \to 0^+} \cE_g \ab( u + \psi_t ) = \cE_g(u).
\]
Since $u \in C^\infty_c(\R^d) \cap \mathcal{A}_{m_1}$ is arbitrary, 
$\wt{e} (g,m_2) \leq \wt{e} (g,m_1)$ holds. 
%
\end{proof}

\begin{proof}[Proof of \cref{thm:2ndsol-nonex} assertion 2]
Let $\omega_0 \in S_m$ be a minimizer of $\cE_{g^*_d(m)}$.
Assume that $\Vab{\nabla \omega_0}_2^2 =r^2\leq \rho_0^2$ where $\rho_0>0$ is the number in \cref{l:wall}: 
then \cref{l:wall} yields 
\begin{equation*}
0=\cE_{g^*_d(m)}= \inf \Set{ \cE_g(u) | u \in S_{m}, \ \Vab{\nabla u}_2^2 =  r^2} \geq \frac{r^2}{8} >0
\end{equation*}
a contradiction. Thus, $\Vab{\nabla \omega_0}_2^2 > \rho_0^2$.
Similarly, since $\cE_{g^*_d(m)} (\omega_0) = 0$, we have $\cE_{g} (\omega_0)=\frac{1}{2}\Vab{\nabla \omega_0}_2^2\left(1-\frac{g}{g^*_d(m)}\right)$  and 
there exists $g_3 \in (0, g^*_d(m) )$ such that $\cE_g(\omega_0) < \rho_0^2/33$ holds for all $g \in [g_3,g^*_d(m))$.  
Consequently, 
\[
\wt{e} (g,m) \leq \cE_g( \omega_0 ) < \frac{\rho^2_0}{33} \quad \text{for every $g \in [g_3,g^*_d(m))$}.
\]
In particular, the first inequality implies
\begin{equation*}\label{eq:local_to_global}
\lim_{g\to g^*_d(m)} \wt{e} (g,m)=0
\end{equation*}
which again is reminiscent of the typical behavior of the Gibbs functional upon approaching the transition point, where the previously metastable states become global equilibrium states \cite{Callen}. 

Let $(u_n)_{n=1}^\infty \subset \mathcal{A}_m$ be any minimizing sequence for $\wt{e} (g,m) $.
By \cref{prop1}, $(u_n)_{n=1}^\infty$ is bounded in $H^1(\R^d)$. 
If $\sup_{z \in \mathbb{Z}^d} \Vab{u_n}_{L^2(z+Q)} \to 0$, 
then $\Vab{u_n}_q \to 0$ for all $2 < q < 2^*$ and $\cV(u_n) \to 0$. 
Therefore, 
\[
 \frac{\rho^2_0}{33} > \cE_g (\omega_0) \geq \wt{e} (g,m) = \lim_{n \to \infty} \cE_g(u_n) 
 = \lim_{n \to \infty} \frac{1}{2} \Vab{\nabla u_n}_2^2.
\]
On the other hand, by $u_n \in \mathcal{A}_m$, $\Vab{\nabla u_n}_2^2 > \rho_0^2 /4$, which is a contradiction. 
As a consequence, $\lim_{n \to \infty} \sup_{z \in \mathbb{Z}^d} \Vab{u_n}_{L^2(z+Q)} = c_0>0$.

Choose $(z_n)_{n=1}^\infty \subset \mathbb{Z}^d$ so that $\Vab{u_n}_{L^2(z_n+Q)} \to c_0 > 0$ and 
set $v_n \coloneq u_n(\cdot +x_n)$. Then $\cE_g(v_n) \to \wt{e} (g,m)$ by translation invariance, and, by the Rellich--Kondrashov theorem, 
$v_n \rightharpoonup v_0 \not \equiv 0 $ weakly in $H^1(\R^d)$. 
Our goal is to show that $v_0$ is a local minimizer, i.e. a minimizer of \eqref{def:local_problem}.

The first step is to establish that $\Vab{\nabla v_0}_2^2 > \rho_0^2/4$. 
Since
\begin{equation*}\label{eq:close}
\frac{\rho_0^2}{33} > \wt{e}(g,m) = \cE_g(v_n) + o(1) ,
\end{equation*}
\cref{l:wall} implies $\Vab{\nabla v_n}_2^2 \geq \rho_0^2$. 
Indeed, note that $\Vab{\nabla v_n}_2^2 \geq \frac{\rho_0^2}{4}$ a priori, uniformly in $g$. If $\Vab{\nabla v_n}_2^2<\rho_0^2$, then 
we have by \cref{l:wall}
\[
	\frac{\rho_0^2}{33}>\wt{e}(g,m)=  \cE_g(v_n) + o(1)\geq \frac{1}{8} \Vab{\nabla v_n}_2^2 + o(1) \geq \frac{\rho_0^2}{32}+o(1),
\]
a contradiction.
Suppose that $\Vab{\nabla v_0}_2^2 \leq \rho_0^2/4$ holds; 
then by setting $w_n \coloneq v_n - v_0$, 
\[
\rho_0^2 \leq \Vab{\nabla v_n}_2^2 = \Vab{\nabla v_0 }_2^2 + \Vab{\nabla w_n}_2^2 + o(1) 
\leq \frac{\rho_0^2}{4} + \Vab{\nabla w_n}_2^2 + o(1).
\]
This leads to $\Vab{\nabla w_n}_2^2 > \rho_0^2 / 2$ and $w_n \in \mathcal{A}_{ \Vab{w_n}_2^2 }$ for all sufficiently large $n$. 
From $v_0 w_n \to 0$ strongly in $L^q(\R^d)$ for each $q \in [1,2^*/2)$, it follows that 
\[
\cV(v_n) = \cV(v_0 + w_n) = \cV(v_0) + \cV(w_n) + o(1), 
\]
and hence 
\[
\wt{e} (g,m) = \cE_g(v_n) + o(1) = \cE_g(v_0) + \cE_g(w_n) + o(1).
\]
Recalling $ \Vab{v_0}_2^2 \leq m$ yields $\cE_g(v_0) \geq e(g, \Vab{v_0}_2^2  ) \geq e(g, m ) = 0$. 
Consider now two cases: if $\Vab{v_0}_2^2 = m$, then $\Vab{w_n}_2^2 \to 0$ and $\cV (w_n) \to 0$ hold. 
Thus, the fact $\Vab{\nabla w_n}_2^2 > \rho_0^2 / 2$ gives 
\[
\frac{\rho_0^2}{33} > \wt{e}(g,m) \geq \liminf_{n \to \infty} \cE_g (w_n) 
= \liminf_{n \to \infty} \frac{1}{2} \Vab{\nabla w_n}_2^2 \geq \frac{\rho_0^2}{4},
\]
which is absurd.

On the other hand, when $ 0 < \Vab{v_0}_2^2 < m$, then as $\Vab{w_n}_2^2 \to m - \Vab{v_0}_2^2$, 
\cref{Lem:mono} and $w_n \in \mathcal{A}_{ \Vab{w_n}_2^2 }$  lead to 
\[
\wt{e} (g,m)  \geq \cE_g (v_0) +  \liminf_{n \to \infty} \cE_g(w_n) 
\geq  
\cE_g (v_0) + \wt{e} (g, m - \Vab{v_0}_2^2 )
\geq 
\cE_g (v_0) + \wt{e} (g, m ). 
\]
By $g_3 \leq g < g^*_d(m)$ and $0 \geq e(g, \Vab{v_0}_2^2 ) \geq e(g,m) = 0$, 
we see that $g^*_d( \Vab{v_0}_2^2 ) \geq g^*_d(m)$, $\cE_g(v_0) = 0$ and 
$v_0$ is a global minimizer. However, this contradicts the nonexistence result in \cref{thm}. 
Thus, the inequality  $\Vab{\nabla v_0}_2^2 > \rho_0^2/4$ follows.

Noting that the splitting 
\[
\cE_g(v_n) = \cE_g(v_0) + \cE_g(w_n) + o(1)
\]
holds, by $\liminf_{n \to \infty} \cE_g(w_n) \geq e(g, m - \Vab{v_0}_2^2 ) \geq e(g,m) = 0$, 
we see from 
the fact $v_0 \in \mathcal{A}_{\Vab{v_0}_2^2}$ and \cref{Lem:mono} that 
\[
\wt{e} (g,m) \geq \cE_g (v_0) +  \liminf_{n \to \infty} \cE_g(w_n) \geq \cE_g(v_0) \geq \wt{e} ( g,  \Vab{v_0}_2^2  ) 
\geq \wt{e} (g,m).
\]
Hence, 
\begin{equation}\label{eq:tilde-e}
	\wt{e} (g,m) = \wt{e} (g, \Vab{v_0}_2^2 ) = \cE_g( v_0 ). 
\end{equation}
Thus, to conclude, it remains to show that $\Vab{v_0}_2^2 = m$. \eqref{eq:tilde-e} shows that $v_0$ minimizes $\mathcal{E}_g$ on the set $\mathcal{A}_{\Vab{v_0}_2^2}$, and it is thus a critical point. This implies that there exists $\lambda_0 \in \R$ such that 
\[
-\Delta v_0 + \lambda_0 v_0 = \frac{g}{2} \Bab{ V * v_0^2 + \check{V} * v_0^2 } v_0 \quad \text{in}  \ \R^d, \quad 
\check{V} (x) \coloneq V(-x). 
\] 
When $\lambda_0 \neq 0$, it follows that 
\begin{equation}\label{eq:deri-Eg}
\odv*{ \cE_g ( t v_0 ) }{t}_{t=1} = \Vab{\nabla v_0}_2^2 - g \int_{\R^d} (V*v_0^2) v_0^2 \odif{x} 
= - \lambda_0 \Vab{v_0}_2^2 \neq 0.
\end{equation}
Suppose that $\|v_0\|_2^2<m$. If $\lambda_0 > 0$, then by $\Vab{\nabla v_0}_2^2 > \rho_0^2/4$, we notice that 
$\Vab{(1+\e) v_0 }_2^2 < m$ and $(1+\e) v_0 \in \mathcal{A}_{ (1+\e)^2 \Vab{v_0}_2^2 }$ for sufficiently small $\e>0$. 
\cref{Lem:mono} with \eqref{eq:tilde-e} leads to
\[
\wt{e} (g,m) \leq \wt{ e} ( g, (1+\e)^2 \Vab{v_0}_2^2 ) \leq  \cE_g( (1+\e) v_0  ) < \cE_g( v_0 ) = \wt{e} (g,m),
\]
which is a contradiction. 
In a similar way, if $\lambda_0 < 0$, then 
$(1-\e) v_0 \in \mathcal{A}_{ (1-\e)^2 \Vab{v_0}_2^2 }$ and \eqref{eq:deri-Eg} give 
\[
\wt{e} (g, m ) \leq \wt{e} (g, \Vab{(1-\e) v_0}_2^2) 
\leq \cE_g( (1-\e) v_0 ) < \cE_g(v_0) = \wt{e} (g, m ),
\]
which is absurd. 
Therefore, $\lambda_0 = 0$ holds. However, in this case, note that 
\[
\odv*[order={1}]{\cE_g(tv_0) }{t}_{t=1} = 0, \quad 
\odv*[order={2}]{\cE_g(tv_0) }{t}_{t=1} 
= \Vab{\nabla v_0}^2_2 - 3g \int_{\R^d} (V*v_0^2) v_0^2 \odif{x} = - 2 \Vab{\nabla v_0}_2^2 < 0
\]
Hence, we can argue as in the case $\lambda_0 > 0$ and obtain a contradiction. 
Thus, $\Vab{v_0}_2^2 = m$ holds and we see from \eqref{eq:tilde-e} that $v_0$ is the desired minimizer. 
Since $\Vab{v_n}_2^2 \to \Vab{v_0}_2^2 = m$, we can show that $\Vab{v_n-v_0}_2^2 \to 0$. 
Moreover, from $\cV(v_n) \to \cV(v_0)$ and 
\[
\wt{e} (g,m) = \lim_{n \to \infty} \cE_g(v_n) \geq \cE_g(v_0) \geq \wt{e} (g,m),
\]
it follows that $\Vab{\nabla v_n}_2^2 \to \Vab{\nabla v_0}_2^2$. 
Therefore, $\Vab{\nabla v_n - \nabla v_0}_2^2 \to 0$ and 
$v_n \to v_0$ strongly in $H^1(\R^d)$. 
This completes the proof .
\end{proof}

\section*{Acknowledgements}
This work was initiated at the conference ``Variational methods in applications to PDEs" held in B\c{e}dlewo near Poznań, Poland,
and the authors would like to express their gratitude to the organizers of the conference. 
They are grateful to B. Bieganowski and R. Seiringer for very important insights. 
It is also a pleasure to thank K. Jachymski and K. Jodłowski for discussions on the physics background. 
The authors would like to thank Goro Akagi and Rikuya Kakinuma for pointing their attention to \cite{JeLu22a}. 
This work was supported by JSPS KAKENHI Grant Number JP24K06802.

\subsection*{Data availability statement} No data was collected nor generated during this study.

\appendix

\section{Computations on the case $d=1$}\label{Sec7}

Here we give an example of $V$ with \ref{A1} and \ref{A2}, and $g_d^*(m) > 0$ holds for every $m > 0$ when $d=1$: 

\begin{proposition}\label{prop:1d}
	Suppose $d=1$. There exists $V$ with {\rm \ref{A1}} and {\rm \ref{A2}} such that 
	$g^*_1(m) > 0$ holds for any $m > 0$. 
\end{proposition}

\begin{proof}
Let $m > 0$ be given, fix $\e \in (0,1/4)$ and consider 
\[
V_\e (x) := \begin{dcases}
	1 & (\vab{x} \leq \e), \\
	-2 & (\e < \vab{x} < 1),\\
	0 & (1 \leq \vab{x}). 
\end{dcases}
\]
Notice that $V_\e$ satisfies \ref{A1} and \ref{A2}. To show $g^*_1(m) > 0$, 
it is enough to prove that for any sufficiently small $g>0$, 
\begin{equation*}
	\cE_g(u) > 0 \quad \text{for all $u \in S_m$}.
\end{equation*}
By the definition of $V_\e$, for each $u \in S_m$, 
\[
- \cV (u) = - \frac{1}{4} \int_{\R} dx \int_{x-\e}^{x+\e} u^2 (x) u^2(y) dy + \frac{1}{2} \int_{\R} dx \int_{\e < \vab{y-x} < 1  } u^2(x) u^2(y) dy. 
\]
For the second term, from $0< \e < 1/4$ it follows that 
\[
\begin{aligned}
	\frac{1}{2} \int_{\R} dx \int_{\e < \vab{y-x} < 1} u^2 (x) u^2(y) dy 
	&\geq 
	\frac{1}{2} \int_{\R} dx \int_{x+\e}^{x+3\e} u^2(x) u^2(y) dy 
	\\
	&= 
	\frac{1}{2} \int_{\R} dx \int_{x-\e}^{x+\e} u^2(x) u^2( s + 2\e ) ds. 
\end{aligned}
\]
Since 
\[
u(s+2\e) = u(s) + 2\e \int_0^1 u'( s + 2\theta \e ) d \theta, 
\]
we obtain 
\[
\begin{aligned}
	\ab( u(s+2\e) )^2 
	&\geq 
	u^2(s) + 4\e u(s) \int_0^1 u'(s+2\theta \e) d \theta
	\\
	&\geq u^2(s) - 2 \e \Bab{ u^2(s) + \ab( \int_0^1 u'(s+2\theta \e) d \theta )^2 } 
	\\
	&\geq (1-2\e) u^2(s) - 2\e \int_0^1 \ab( u'(s + 2\theta \e) )^2 d \theta. 
\end{aligned}
\]
Thus, 
\[
\begin{aligned}
	&\frac{1}{2} \int_{\R} dx \int_{\e < \vab{y-x} < 1} u^2 (x) u^2(y) dy
	\\ 
	\geq \ &
	\frac{1}{2} \int_{\R} dx u^2(x) \int_{x-\e}^{x+\e} \ab[ (1-2\e) u^2(s) -2\e \int_0^1 \ab( u'(s+2\theta \e) )^2 d \theta ] ds
	\\
	\geq \ &
	\frac{1}{2}(1-2\e) \int_{\R} dx u^2(x) \int_{x-\e}^{x+\e} u^2(s) ds - \e \int_{\R} dx u^2(x) \int_{\R} ds 
	\int_0^1 \ab( u'(s+2\theta \e) )^2 d \theta 
	\\
	= \ & 
	\frac{1}{2}(1-2\e) \int_{\R} dx  \int_{x-\e}^{x+\e} u^2(x) u^2(y) dy -  \e \Vab{u}_{2}^2 \Vab{u'}_2^2,
\end{aligned}
\]
where Fubini's theorem is used in the last step. 
Recalling $\Vab{u}_2^2 = m$, we have 
\[
\begin{aligned}
	\cE_g (u) &= \frac{1}{2} \Vab{u'}_2^2 - g \cV(u)
	\\
	&\geq \frac{1}{2} \Vab{u'}_2^2 
	+ \frac{1 - 4\e}{4} g
	\int_{\R} dx u^2(x) \int_{x-\e}^{x+\e} u^2(y) dy - \e g m \Vab{u'}_2^2. 
\end{aligned}
\]
Since $0<\e<1/4$, if $g>0$ satisfies $\e gm < 1/2$, then $\cE_g(u) > 0$ holds for each $u \in S_m$ 
and $g^*_1(m) > 0$ follows for every $m>0$. 
\end{proof}

\subsection*{Conflict of interest} The authors have no Conflict of interest to declare for this article.

\end{document}